\documentclass[onefignum,onetabnum]{siamart171218}

\usepackage{lipsum}
\usepackage{amsfonts}
\usepackage{graphicx,xspace}
\usepackage{epstopdf}
\usepackage{algorithmic}
\usepackage{amssymb} 
\usepackage{amsmath} 
\usepackage{mathtools} 
\usepackage{enumitem} 	
\usepackage{upgreek} 
\usepackage{booktabs} 
\ifpdf
  \DeclareGraphicsExtensions{.eps,.pdf,.png,.jpg}
\else
  \DeclareGraphicsExtensions{.eps}
\fi

\newtheorem{example}{Example}
\newtheorem{assumption}{Assumption}

\newsiamremark{remark}{Remark}

\headers{Continuous-Time Rates in \\ Potential and Monotone Games}{Bolin Gao and Lacra Pavel}

\title{Continuous-Time Convergence Rates in Potential and Monotone Games \thanks{Submitted to the editors February 2, 2022.}
\funding{This work was supported by a grant from NSERC and Huawei Technologies Canada.}}

\author{Bolin Gao\thanks{Department of Electrical and Computer Engineering, University of Toronto, Toronto, ON, M5S 3G4, Canada. \email{bolin.gao@mail.utoronto.ca}}
\and Lacra Pavel\thanks{Department of Electrical and Computer Engineering, University of Toronto, Toronto, ON, M5S 3G4, Canada. \email{pavel@control.utoronto.ca}}
}

\usepackage{amsopn}

\DeclareMathOperator{\dom}{dom}

\DeclareMathOperator{\interior}{int}

\DeclareMathOperator{\rinterior}{rint}

\DeclareMathAlphabet{\mathpzc}{OT1}{pzc}{m}{it} 

\ifpdf
\hypersetup{
	pdftitle={Continuous-Time Rates in \\ Potential and Monotone Games},
	pdfauthor={Bolin Gao and Lacra Pavel}
}
\fi

\graphicspath{{./fig/}}

\newcommand\oprocendsymbol{\hbox{$\square$}}
\newcommand\oprocend{\relax\ifmmode\else\unskip\hfill\fi\oprocendsymbol}

\renewcommand{\theenumi}{(\roman{enumi})}

\DeclareSymbolFont{bbold}{U}{bbold}{m}{n}
\DeclareSymbolFontAlphabet{\mathbbold}{bbold}

\begin{document}

\maketitle

\begin{abstract}
  In this paper, we provide exponential rates of convergence to the interior Nash equilibrium for continuous-time dual-space game dynamics such as mirror descent (MD) and actor-critic (AC). We perform our analysis in $N$-player continuous concave games that satisfy certain monotonicity assumptions while possibly also admitting potential functions. In the first part of this paper, we provide a novel \textit{relative} characterization of monotone games and show that MD and its discounted version converge with  $\mathcal{O}(e^{-\beta t})$ in relatively strongly and relatively hypo-monotone games, respectively. In the second part of this paper, we specialize our results to games that admit a \textit{relatively} strongly concave potential and show that AC converges with $\mathcal{O}(e^{-\beta t})$.  These rates extend their known convergence conditions. Simulations are performed which empirically back up our results.
\end{abstract}

\begin{keywords}
  Potential Game, Monotone Game, Mirror Descent, Actor-Critic, Rate of Convergence, Multi-Agent Learning
\end{keywords}

\section{Introduction}

Due to an ever-increasing number of applications that rely on the processing of massive amount of data, e.g., \cite{Bottou18, Daskalakis18, Goodfellow14, Zhang19}, the rate of convergence has became a paramount concern in the design of algorithms. A prominent line of recent work involves \textit{analyzing the rate of convergence of ordinary differential equations (ODE) via Lyapunov analysis} in order to characterize and enhance the rates of their discrete-time counter-parts, e.g., \cite{Orecchia19, Krichene15, Raginsky12, Su}. For instance, in the \textit{primal-space} where the iterates are directly updated, it is known that the function value along the (time-averaged) trajectory of gradient flow converge to the optimum of convex problems in $\mathcal{O}(1/t)$\footnote{Recall that given $f, g: \mathbb{R} \to \mathbb{R}, f(t) \in \mathcal{O}(g(t))$ (or $f(t) = \mathcal{O}(g(t))$) if $\exists M, T > 0$, such that $|f(t)| \leq Mg(t), \forall t > T.$ }\cite{Krichene15, Su}. In the \textit{dual-space}, whereby the gradient is processed and mapped back through a \textit{mirror operator}, a $\mathcal{O}(1/t)$ rate was shown for continuous-time mirror descent \cite{Krichene15, Raginsky12}. These rates were later improved to $\mathcal{O}(1/t^2)$ through the design of non-autonomous ODEs \cite{Krichene15, Su}.

Although convergence rates have been thoroughly studied in the optimization setup, a similar analysis for the analogous continuous game setting, particularly for $N$-player games with continuous-time game dynamics, has been far less systematic. This could be due to several key distinctions between the two setups:
\begin{enumerate}[leftmargin = *]
	\item While in the optimization framework, the global optimum (or the function value at the optimum) is the target for which the rate of convergence is measured, in a game, there can exist multiple desirable solution concepts, e.g.,  dominant strategies, pure, mixed, along with various notions of perturbed Nash equilibrium and their refinements \cite{Weibull}. Hence in any given game there can exist multiple metrics and targets for which the rate is measured. 
	\item Unlike the optimization setup, the rate metric depends on multiple payoff/cost functions instead of a single one. Even for the simpler setting of two-player zero-sum games, each player's payoff function is a saddle function (e.g., convex in one argument, concave in the other). Therefore many useful properties widely employed in optimization-centric rate analysis cannot be applied to the whole argument of any player's individual payoff function.
	\item The issue of \textit{convergence}, on which the rate analysis necessarily rests upon, is also more complex. It is well-known that even the most prototypical dynamics for games such as (pseudo-)gradient and mirror descent dynamics may not necessarily converge \cite{Hart03} and can cycle in perpetuity for zero-sum games \cite{Mertikopoulos18}. Outside of zero-sum games, e.g., certain Rock-Paper-Scissors games with non-zero diagonal terms, game dynamics can exhibit limit cycles or chaos \cite{Sato}. Hence, the convergence as well as the rate for which these dynamics achieve must be qualified in terms of more complex properties that characterize the entire set of payoff functions.		
\end{enumerate} 

Motivated by these questions, in this paper we characterize the rate of convergence for two general families of dual-space dynamics towards the interior Nash equilibrium (or Nash-related solutions) of $N$-player continuous concave games, specifically, in games that satisfy certain monotonicity conditions, which may also possess potential functions. These games are referred to as monotone games and potential games, respectively. Prototypical examples of potential games include standard formulations of Cournot games, symmetric quadratic games, coordination games, and flow control games, whereas monotone games capture examples of network zero-sum games, Hamiltonian games, asymmetric quadratic games, various applications arising from networking and machine learning such as generative adversarial networks (GAN) \cite{Daskalakis18, Goodfellow14}, adversarial attacks \cite{Zhang19}, as well as mixed-extension of finite games, such as Rock-Paper-Scissors and Matching Pennies \cite{Basar12, Bo_LP_CDC2020, Bo_LP_TAC2020}.

{\textit{Literature review.}} We provide a non-exhaustive survey of the rates of convergence of continuous game dynamics in $N$-player continuous games prior to our work. We broadly divide these results in two prominent settings: those belonging to mixed-strategy extension of finite normal-form games, hereby referred to as \textit{mixed games} for brevity, and more general types of continuous games. For similar discussion in the related framework of population games, see recent work such as \cite{ochoa2019hybrid}.

For mixed games, the earliest works showed that the rate of elimination of strictly dominated strategies (which can be thought of as the rate of divergence) in $N$-player mixed games for ($n^\text{th}$-order variant) replicator dynamics is exponential \cite{Laraki13, Weibull}, which was generalized by \cite{Mertikopoulos16} for dynamics arising from alternative choices of regularizers. The rates for continuous-time fictitious play and best-response dynamics in regular (exact) potential games were shown to be exponential in \cite{Swenson17FP, Swenson18BR}. Exponential convergence was shown for continuous-time fictitious and gradient play with \textit{derivative action} in \cite{Shamma05} for two-player games, which is generalizable to $N$-players games. Except for \cite{Mertikopoulos16}, all of these works involve strategy updates in the primal-space. In the dual space, a general result by \cite{Mertikopoulos18} showed that the continuous-time Follow the Regularized Leader minimizes the \textit{averaged regret} in mixed games with rate $\mathcal{O}(1/t)$.

For continuous games beyond mixed games, early work by \cite{Flam90} provided exponential convergence of continuous-time projected subgradient dynamics in a (restricted) strongly monotone game. Exponential convergence was shown for projected gradient dynamics under full and partial information setups in \cite{Dian_Pavel_TAC2019}. Exponential stability of NE can also be shown for various continuous-time dynamics, such as extremum-seeking dynamics \cite{Basar12}, gradient-type dynamics with consensus estimation \cite{Ye17}, affine nonlinear dynamics \cite{Huang20}, among others. We note that all the dynamics listed above are in the primal-space, whereby these rates are characterized in the Euclidean sense. Furthermore, all the authors \cite{Basar12, Huang20, Ye17} place a standard strict diagonal dominance condition on the game's Jacobian at the NE. For dual-space dynamics, \cite{mertikopoulos2017convergence} showed that dual averaging (or lazy mirror descent) converges towards a globally strict variationally stable state in $\mathcal{O}(1/t)$ in terms of an \textit{average equilibrium gap}, which also holds in strictly monotone games.

{\textit{Contributions.}}
Our work provides a systematic Lyapunov-based method for deriving continuous rate of
convergence of dual-space dynamics towards interior Nash-type solutions in $N$-player games in terms of the actual sequence of play. We investigate two general classes of dual-space dynamics, namely, mirror descent (MD) \cite{Mertikopoulos18, Mertikopoulos16, mertikopoulos2017convergence} (and its \textit{discounted} variant \cite{Bo_LP_TechNote19}) and actor-critic (AC) dynamics (closely related to \cite{Krichene15, Leslie03, Perkins17}). We go beyond the classical proof of exponential convergence (or stability) by providing a precise characterization of the rate's dependency on the game's inherent geometry as well as the player's own parameters. As such, we provide theoretically justified  reasonings for choosing between these dynamics based on their rates of convergence.

Our work is divided into two parts. First, we consider monotone games and provide novel \textit{relative} notions of monotonicity. Under this new characterization, we provide exponential rates of convergence for MD and its discounted variant in all major regimes of monotone games, which extends their previously known convergence conditions \cite{Bo_LP_TechNote19, Bo_LP_TAC2020, mertikopoulos2017convergence} in terms of the actual iterates. In the second part, we specialize our results to a potential game setup and provide exponential rates of convergence of AC in relatively strongly concave potential games. An abridged version of the above results in potential games can be found in \cite{Bo_LP_CDC2020}, but without proofs. The commonality of our approach in the potential and monotone games involves making use of the properties of the game's pseudo-gradient, as well as exploiting non-Euclidean generalizations of convexity and monotonicity. This lends generality to our results and enables us to provide the rate of convergence towards interior solutions for all dynamics studied in \cite{Cherukuri17, Coucheney15, Bo_LP_TechNote19, Mertikopoulos18} and for some of the dynamics studied in \cite{Bo_LP_TAC2020, Laraki13, Mertikopoulos16}. In contrast to \cite{Flam90, Dian_Pavel_TAC2019, Shamma05, Swenson17FP, Swenson18BR}, where exponential rates were shown in the Euclidean sense, our analysis uncovers the exact parameters that affect such rates in a non-Euclidean setup, and provides rates both in terms of Bregman divergences and Euclidean distances. In contrast to \cite{Mertikopoulos18, mertikopoulos2017convergence}, we provide convergence rate of the actual iterates as opposed using either time-averaged regret or equilibrium gaps.

Finally, we remind the reader that our results presented here should not be taken as indicative of the rates associated with their discrete-time counter-parts. Indeed, as discussed in \cite{Krichene15}, multiple discretization schemes can correspond to a single ODE, and not all preserve the continuous-time rate. Furthermore, the choice of step-sizes, absent in our analysis, is also crucial in determining the rates of discrete-time algorithms \cite{Cohen17_Hedge}. For additional rate analyses performed in discrete-time in continuous game setups, refer to  \cite{Azizian20, Kadan, Cohen17, Nedic18, mertikopoulos2017convergence, Mokhtari2020, Li19, Wang18}.

{\textit{Paper organization.}}
This paper is organized as follows. In Section~\ref{sec:review_concepts}, we provide the preliminary background. 
Section~\ref{sec:Construction_of_Game_Dynamics} discusses MD (and its discounted variant called DMD) and AC.  Section~\ref{sec:monotone_games} introduces relatively strong and relatively hypo-monotone games and provides rates of convergence for MD and DMD in these regimes. In Section~\ref{sec:relatively_strongly_concave}, we provide rate results for relatively strongly concave potential games for AC. Numerical simulations are presented in Section~\ref{sec:case_studies}. Section~\ref{sec:conclusions} provides the conclusion followed by  Section~\ref{sec:appendix} which contains the proofs of all results.

\section{Review of notation and preliminary concepts}
\label{sec:review_concepts}
\hfill\\
\indent {\textit{Convex Sets, Fenchel Duality, and Monotone Operators}} The following is from \cite{Beck17, Facchinei_I, Rockafellar}. Given a convex set $\mathcal{C} \!\subseteq \!\mathbb{R}^n$, the (relative) interior of $\mathcal{C}$ is denoted as ($\rinterior(\mathcal{C})$) $\interior(\mathcal{C})$.
$\rinterior(\mathcal{C})$ coincides with $\interior(\mathcal{C})$ whenever $\interior(\mathcal{C}) \neq \varnothing$.  $\textstyle \pi_{\mathcal{C}}(x) \!=\!  \text{argmin}_{y \in \mathcal{C}} \|y - x\|_2^2$ denotes the Euclidean projection of $x$ onto $\mathcal{C}$. The simplex in $\mathbb{R}^n$ is denoted as $\Updelta^n = \{ x\in \mathbb{R}^n| \sum_{i = 1}^n x_i = 1, x_i \geq 0\}$. The normal cone of a convex set $\mathcal{C}$ at $x \in \mathcal{C}$ is defined as $N_\mathcal{C}(x) =\{v \in \mathbb{R}^n| v^\top(y-x) \leq 0, \forall y \in \mathcal{C}\}$.  Let $\mathbb{E} = \mathbb{R}^n$ be endowed with norm $\|\cdot\|$ and inner product $\langle \cdot, \cdot \rangle$. An extended real-valued function is a function $f$ that maps from $\mathbb{E}$ to $[-\infty, \infty]$. The (effective) domain  of $f$ is $\dom(f) = \{x \in \mathbb{E} : f(x) < \infty\}$. 
A function $f: \mathbb{E} \to [-\infty, \infty]$ is proper if it does not attain the value $-\infty$ and there exists at least one $x \in \mathbb{E}$ such that $f(x)  < \infty$; it is closed if its epigraph is closed. Given $f$, the function $f^\star \!: \!\mathbb{E}^\star \! \to \! [-\infty, \infty]$ defined by $f^\star(z) \!=\! \sup_{x  \in \mathbb{E}}  \big[x^\top z - f(x)\big]$, 
is called the convex conjugate of $f$, where  $\mathbb{E}^\star\!$ is the dual-space of $\mathbb{E}$, endowed with the dual norm $\| \cdot\|_\star $.   $f^\star$ is closed and convex if $f$ is proper.  Let  $\partial f(x)$ denote a subgradient of $f$ at $x$ and  $\nabla f(x)$ the gradient of $f$ at $x$, if $f$ is differentiable.
The Bregman divergence of a proper, closed, convex function $f$, differentiable over $\dom(\partial f)$, is $D_f \!:\! \dom(f) \!\times \!\dom(\partial f) \to \mathbb{R}, D_f(x,y) \!=\! f(x) \!-\! f(y) \!-\! \nabla f(y)^\top(x\!-\!y)$, where $\dom(\partial f) = \{x \in \mathbb{R}^n | \partial f(x) \neq \varnothing\}$ is the effective domain of $\partial f$.  $F \!: \!\mathcal{C} \! \subseteq \!\mathbb{R}^n \!\to \!\mathbb{R}^n$ is monotone  if $(F(z) \!-\! F(z^\prime))^\top (z-z^\prime) \!\geq\! 0$, $\forall z,z^\prime \! \in \! \mathcal{C}$.
$F$ is $L$-Lipschitz on $\mathcal{C}$ if $\|F(z) \! -\! F(z^\prime)\|_2 \!\leq  \!L \|z \!-\! z^\prime\|_2$, $\forall z,z^\prime \! \in \! \mathcal{C}$, for some $L \!>\! 0$. Suppose $F$ is the gradient of a scalar-valued function $f$, then $f$ is $\ell$-smooth if $F$ is $\ell$-Lipschitz.
The Jacobian of $F$ is denoted as $\mathbf{J}_F$.

{\textit{$N$-Player Continuous Concave Games} Let $\mathcal{G} = (\mathcal{N}, \{\Omega^p\}_{p \in \mathcal{N}}, \{\mathcal{U}^p\}_{p\in \mathcal{N}})$ be a game, where  $\mathcal{N} = \{1, \ldots, N\}$ is the set of players, $\Omega^p \subseteq \mathbb{R}^{n_p}$ is the set of player $p$'s strategies. We denote the strategy set of player $p$'s opponents as $\Omega^{-p} \subseteq \prod_{q \in \mathcal{N}, q \neq p} \mathbb{R}^{n_q}$ and the set of all the players strategies as $\Omega =   \prod_{p \in \mathcal{N}} \Omega^{p} \subseteq \prod_{p \in \mathcal{N}} \mathbb{R}^{n_p} = \mathbb{R}^{n}, n = \sum_{p \in \mathcal{N}} n_p$. We refer to  $\mathcal{U}^p: \Omega \to \mathbb{R}, x \mapsto \mathcal{U}^p(x)$, as player $p$'s real-valued payoff function, where $x  = (x^p)_{p \in \mathcal{N}} \in \Omega$ is the action profile of all players, and $x^p \in \Omega^p$ is the action of player $p$. We also  denote $x$ as $x  = (x^p;x^{-p})$ where $x^{-p} \in \Omega^{-p}$ is the action profile of all players except $p$. For differentiability purposes, we make the implicit assumption that there exists some open set, on which $\mathcal{U}^p$ is defined and continuously differentiable, such that it contains $\Omega^p$. 
	\begin{assumption} For all $p \in \mathcal{N}$,
		\label{assump:blanket}
		$\Omega^p$ is a non-empty, compact, convex, subset of $\mathbb{R}^{n_p}$, $\mathcal{U}^p(x^p;x^{-p})$ is (jointly) continuous in $x = (x^p;x^{-p})$, $\mathcal{U}^p(x^p;x^{-p})$ is concave and continuously differentiable in each $x^p$ for all $x^{-p} \in \Omega^{-p}$. 
	\end{assumption} 
	Under \cref{assump:blanket},  $\mathcal{G}$ is a \textit{continuous (concave) game}. 
	Given $x^{-p} \in \Omega^p$, 
	each agent $p \in \mathcal{N}$ aims to find the solution of the following optimization problem,  
	\begin{equation}
		\begin{aligned}
			& \underset{x^p}{\text{maximize}}
			& &  \mathcal{U}^p(x^p; x^{-p})
			& \text{subject to}
			& & x^p \in \Omega^p.
		\end{aligned}
	\end{equation}
 	
	A profile 	${{x}}^\star \!=\! ({{x}^p}^\star)_{p \in \mathcal{N}} \!\in \!\Omega$  is a Nash equilibrium (NE) if, 
	\begin{equation}\label{Nash_definition}
		\mathcal{U}^p({x^p}^\star; {x^{-p}}^\star) \geq \mathcal{U}^p(x^p; {x^{-p}}^\star), \forall x^p \in \Omega^p, \forall p \in \mathcal{N}.
	\end{equation}
	At a NE, no player can increase his payoff by unilateral deviation. Under \cref{assump:blanket}, the existence of a NE is guaranteed \cite[Theorem 4.4]{Basar}. 
	
	A useful characterization of a NE of a concave game $\mathcal{G}$ is in terms of the \textit{pseudo-gradient}, $U: \Omega \to \mathbb{R}^n, U(x) \!=\!(U^p(x))_{p \in \mathcal{N}}$, where 
	$U^p(x) =\nabla_{x^p} \mathcal{U}^p(x^p;x^{-p})$ is the \textit{partial-gradient} of player $p$ \footnote{When $\interior(\Omega^p) = \varnothing$, $U^p$ is the partial-gradient of $\mathcal{U}^p$ taken with respect to $\rinterior(\Omega^p) \neq \varnothing$.}. 
	We make the following common regularity assumption.	

	\begin{assumption}
		$U: \Omega \to \mathbb{R}^n$ is $L$-Lipschitz on $\Omega$ (possibly excluding the relative boundary). 
	\end{assumption}

	By \cite[Proposition 1.4.2] {Facchinei_I}, $x^\star \in \Omega$ is a NE if and only if,
	\begin{equation}
		(x-x^{\star} )^\top U(x^{\star}) \leq 0, \forall x \in \Omega. 
		\label{eqn:nash_equilibrium_VI}
	\end{equation}	
	The NE $x^\star$ is said to be \textit{interior} if $x^\star \in \rinterior(\Omega)$.
\begin{definition}
	\label{def:potential_game} 
	A concave game $\mathcal{G}$ is an exact potential game if	there exists a scalar-valued function $P: \Omega \to \mathbb{R}$, referred to as the \textit{potential function}, such that, $\forall p \in \mathcal{N}, x^{-p} \in \Omega^{-p}$, $x^p, x^{p\prime} \in \Omega^p$,
	\vspace{-0.2cm}
	\begin{equation}
		\mathcal{U}^p(x^p; x^{-p})\! - \!\mathcal{U}^p(x^{p\prime}; x^{-p}) \!=\! P(x^p; x^{-p})\! -\! P(x^{p\prime}; x^{-p}).
	\end{equation}
\end{definition} 
From \cref{def:potential_game}, it is clear that $U = \nabla P$  whenever $P$ is differentiable, therefore by \eqref{eqn:nash_equilibrium_VI}, any global maximizer of $P$ is a NE. $\mathcal{G}$ can be shown to be an exact potential game whenever the Jacobian of the pseudo-gradient,  $\mathbf{J}_U(x)$, is symmetric for all $x$ \cite[Theorem 1.3.1, p. 14]{Facchinei_I}.

\section{Dual-Space Game Dynamics}
\label{sec:Construction_of_Game_Dynamics}

\indent In this section, we introduce several dual-space game dynamics that have been previously studied in the continuous game literature. We motivate these dynamics through the following interaction model: suppose a set of players are repeatedly interacting in a game $\mathcal{G}$. Starting from an initial strategy $x^p(0) \in \Omega^p$, each player $p$ plays the game and obtains a partial-gradient $U^p(x) \!\in \! \mathbb{R}^{n_p}$. Each player then maps his own partial-gradient vector $U^p(x)$ into an unconstrained \textit{auxiliary variable} or \textit{dual aggregate} $z^p \! \in \! \mathbb{R}^{n_p}$ via a dynamical process. Then a device known as the \textit{mirror map} $C^p_\epsilon: \mathbb{R}^{n_p} \to \Omega^p$ suggests to the player a strategy, for which the player can either directly use as his next strategy or process it further. The game is then played again using the players' chosen strategies.  The two main classes of dual-space game dynamics which correspond to this setup are the family mirror descent (MD) and actor-critic (AC) dynamics, which we discuss below.

{\textit{Mirror Descent}}
	The most commonly studied class of dual-space dynamics is MD, which consists of the following system of ODEs, 
	\begin{equation}
		\label{eqn:MD}
		\dot z^p = \gamma U^p(x), \quad 	x^p = C^p_\epsilon(z^p),
	\end{equation}
	where $C^p_\epsilon$ is the \textit{mirror map}, $C^p_\epsilon: \mathbb{R}^{n_p}  \to  \Omega^p$,\vspace{-0.23cm}
	\begin{equation} 
		\label{eqn:mirror_map_argmax_char} 
		C^p_\epsilon(z^p) = \text{argmax}_{y^p \in \Omega^p}\left[ {y^p}^\top z^p  - \epsilon \vartheta^p(y^p)\right], \epsilon > 0,
	\end{equation} 
	where $\vartheta^p: \mathbb{R}^{n_p} \to \mathbb{R}^{n_p}\cup\{\infty\}$ is assumed to be a closed, proper and strongly convex function, referred to as a \textit{regularizer} and $\dom(\vartheta^p) = \Omega^p$ is assumed be a non-empty, compact and convex set.
	
	Depending on  $C^p_\epsilon$, \eqref{eqn:MD} captures a wide range of existing game dynamics. For general $C^p_\epsilon$, it represents the continuous-time, game theoretic extension of dual averaging or lazy mirror descent \cite{Duchi, mertikopoulos2017convergence} or Follow-the-Regularized-Leader \cite{Mertikopoulos18}. When $C^p_\epsilon$ is the identity, \eqref{eqn:MD} captures the pseudo-gradient dynamics (for $p = N$) \cite{Bo_LP_TechNote19}, the saddle-point dynamics \cite{Cherukuri17} (for $p = 2$), and the gradient flow (for $p = 1$). When $C^p_\epsilon$ is the \textit{softmax} function, \eqref{eqn:MD} corresponds to exponential learning \cite{Mertikopoulos16}, which induces the replicator dynamics \cite[p. 126]{Sandholm10} as its primal dynamics. 
	
	A closely related set of dynamics is the \textit{discounted} mirror descent (DMD) 	 \cite{Coucheney15, Bo_LP_TechNote19}, 
	\begin{equation}
		\label{eqn:DMD}
		\dot z^p = \gamma (-z^p + U^p(x)), \quad 	x^p = C^p_\epsilon(z^p).
	\end{equation}
	Compared to MD, an extra $-z^p$ term is inserted in the $\dot z^p$ system, which translates into an exponential decaying term in the closed-form solution, i.e., $z^p(t) = e^{{-\gamma}t} z^p(0) + \gamma \smallint\nolimits_0^t e^{{-\gamma}(t-\tau)} U^p(x(\tau)) \mathrm{d}\tau$. DMD is also related to the \textit{weight decay} method in the machine learning literature \cite{hanson1988comparing}, as $z^p \in ({C^p_\epsilon})^{-1}(x^p)$ can be shown to be equivalent to a regularization term, which directly interacts with the monotonicity property of $U^p$.

{\textit{Actor-Critic}} A second class of dual-space dynamics is the family of AC dynamics,
	\begin{equation}
		\label{eqn:AC}
		\dot z^p =  \gamma U^p(x), \quad  
		\dot x^p = r(C^p_\epsilon(z^p) - x^p).
	\end{equation}
	In contrast to MD, AC models the scenario whereby the player further processes the output of $C^p_\epsilon$ \eqref{eqn:mirror_map_argmax_char} through discounted aggregation in the primal-space. AC has been previously investigated in the game and optimization literature. For example, a version of AC with time-varying coefficients known as \textit{accelerated mirror descent} (AMD) was studied in \cite{Krichene15} in a convex optimization setup. In games, AC is related to the continuous-time version of the algorithm by the same name in \cite{Leslie03, Perkins17} and can be seen as a dual-space extension of the \textit{logit dynamics} \cite[p. 128]{Sandholm10}. Differing from MD, for which convergence in strictly monotone games is known (see \cite{mertikopoulos2017convergence}), AC-type dynamics have only been investigated in potential game setups and AC is not known to converge in games that do not admit potentials \cite{Leslie03, Perkins17}.

{\textit{Construction of the Mirror Map}} The definition of the mirror map $C^p_\epsilon$ \eqref{eqn:mirror_map_argmax_char}  is intimately tied to the properties of the regularizer $\vartheta^p$. In this work, we assume that $\vartheta^p$ satisfies the following basic assumption.

\begin{assumption}
	\label{assump:primal}
	The regularizer $\vartheta^p\!:\! \mathbb{R}^{n_p} \!\!\to \!\mathbb{R}\!\cup\!\{\!\infty \!\}$ is closed, proper, $\rho$-strongly convex, with $\dom(\vartheta^p) = \Omega^p $ non-empty, compact, convex.
\end{assumption}

We further classify $\vartheta^p$ into either \textit{steep or non-steep}. $\vartheta^p$ is said to be steep if  $\|\nabla \vartheta^p(x_k^p)\|_\star \to +\infty$ whenever $\{x_k^p\}_{k = 1}^\infty$ is a sequence in $\rinterior(\dom(\vartheta^p))$ converging to a point in the (relative) boundary. It is non-steep otherwise.

In order to properly incorporate the regularization parameter $\epsilon$,
cf. \!\eqref{eqn:mirror_map_argmax_char}, we consider $\psi_\epsilon^p \coloneqq \epsilon \vartheta^p\!$, which inherits all properties of $\vartheta^p\!$.  
We denote the convex conjugate of $\psi_\epsilon^p$ as ${\psi_\epsilon^p}^\star = {\text{max}}_{y^p \in \Omega^p}\left[ {y^p}^\top z^p  - \epsilon \vartheta^p(y^p)\right]$. We then refer to $C^p_\epsilon$ as the \textit{mirror map induced by $\psi_\epsilon^p\!$}. The key properties of $C^p_\epsilon\!$ under \cref{assump:primal} are presented in the Appendix. Next, we introduce two examples of mirror maps.

\begin{example} (Strongly convex, non-steep)
	\label{examp:projection}
	Let $\Omega^p \subset \mathbb{R}^{n_p}$ be nonempty, compact and convex. Consider $ \vartheta^p(x^p) = \frac{1}{2}\|x^p\|^2_2$, $x^p \in \Omega^p$. The mirror map generated by $\psi_\epsilon^p = \epsilon \vartheta^p$ is the \textit{Euclidean projection}, $ C^p_\epsilon(z^p) = \pi_{\Omega^p}(\epsilon^{-1} z^p) = \text{argmin}_{y^p \in \Omega^p} \|y^p - \epsilon^{-1} z^p\|_2^2$.
\end{example} 

\begin{example}(Strongly convex, steep)
	\label{examp:softmax}
	Let $\Updelta^{n_p} =  \Omega^p \! = \{x^p \! \in \!\mathbb{R}^{n_p}_{\geq 0}|\| x^p\|_1 = 1\}$ and $\textstyle \vartheta^p(x^p) =  \sum_{i = 1}^{n_p} x_i^p \log(x_i^p)$ ($0 \log 0 =0$). The mirror map generated by $\psi_\epsilon^p = \epsilon \vartheta^p$ is referred to as \textit{softmax} or \textit{logit map}, $ \textstyle C^p_\epsilon(z^p) \!=\!(\exp(\epsilon^{-1} \!z^p_i)(\sum_{j \! = \!1}^{n_p} \exp(\epsilon^{-1} \!z_j^p))^{-1})_{i \in \{1,\dots,n_p\}}.$
\end{example}

\section{Rate of Convergence in Monotone Games}
\label{sec:monotone_games}

In this section, we discuss the convergence rate of the family of MD dynamics in monotone games. Let's  recall the standard definitions associated with the class of \textit{monotone}  games \cite{Facchinei_I}. 
\begin{definition}
	\label{def:monotone}
	The game $\mathcal{G}$ is,
	\begin{itemize}
		\item [(i)] null monotone if, 
		$(U(x) - U(x^\prime))^\top(x-x^\prime) = 0, \forall x, x^\prime \in \Omega.$
		\item [(ii)] merely monotone if,
		$		(U(x) - U(x^\prime))^\top(x-x^\prime) \leq 0, \forall x, x^\prime \in \Omega.$
		\item[(iii)] strictly monotone if,
		$	(U(x) - U(x^\prime))^{\!\top}\!(x-x^\prime) \!<\! 0, \forall x \! \neq \!x^\prime  \in \Omega.
		$
		\item[(iv)] $\eta$-strongly monotone if, 
		$		\!(U(x) - U(x^\prime))\!^\top\!(x-x^\prime) \!\leq \! -\eta\|x \!-\! x^\prime\|^2_2$, $\! \forall x,x^\prime  \!\in \!\Omega$, for some $\eta \!>\! 0$.
		\item[(v)] $\mu$-hypo-monotone if, 
		$		\!(U(x) - U(x^\prime))\!^\top\!(x-x^\prime) \!\leq \! \mu\|x \!-\! x^\prime\|^2_2$, $\! \forall x,x^\prime  \!\in \!\Omega$, for some $\mu \!>\! 0$.
	\end{itemize} 
\end{definition}\vspace{-0.25cm}
These monotonicity properties have well-known second-order characterizations using the negative semi-definiteness of the Jacobian of $U$ \cite[p. 156]{Facchinei_I}. When $\mathcal{G}$ does not admit a potential function, e.g., $\mathbf{J}_U$ is asymmetric, we refer to  $\mathcal{G}$ as \textit{potential-free}. Next, we introduce a generalization to strong/hypo monotonicity in terms of a \textit{relative} (or reference) function $h$, which helps to capture the geometry associated with the game arising from either the payoff functions or the players' action sets.

\begin{definition}
	\label{def:relative_monotone}
	Let $h: \mathbb{R}^n \to \mathbb{R} \cup \{\infty\}$ be any differentiable, strongly convex function with domain $\dom(h) = \Omega$, then $\mathcal{G}$ is, 
	\begin{itemize}
		\item[(i)] $\eta$-relatively strongly monotone (with respect to $h$) if,  
		$		\!(U(x) \!-\! U(x^\prime))\!^\top\!(x\!-\!x^\prime) \!\leq \! -\eta(D_h(x, x^\prime) + D_h(x^\prime, x))$, $\! \forall x,x^\prime  \!\in \! \dom(\partial h) \subseteq \Omega$, for some $\eta \!>\! 0$.
		\item[(ii)] $\mu$-relatively hypo-monotone  (with respect to $h$) if, 
		$		\!(U(x) \!-\! U(x^\prime))\!^\top\!(x\!-\!x^\prime) \!\leq \! \mu(D_h(x, x^\prime) + D_h(x^\prime, x))$, $\! \forall x,x^\prime  \!\in \! \dom(\partial h) \subseteq \Omega$, for some $\mu \!>\! 0$.
	\end{itemize} 
\end{definition}
For the case where $x\in \Omega \subseteq \mathbb{R}^n$, $h(x) = \frac{1}{2}\|x\|_2^2$ , $D_h(x, x^\prime) = D_h(x^\prime, x) = \frac{1}{2}\|x - x^\prime\|_2^2$, thus relative (strongly/hypo) monotonicity coincides with standard (strongly/hypo) monotonicity. These definitions are restricted to $\dom(\partial h)$ to account for the cases when $h$ is steep, in which case $\dom(\partial h) = \rinterior(\Omega)$.  Next, we provide a result that relates standard and \textit{relative} monotonicity for more general classes of $h$. Recall that a differentiable, convex function $h$ is $\ell$-smooth on $\mathcal{C} \subseteq \mathbb{R}^n$ for $\ell \geq 0$ if $\|\nabla h(x) - \nabla h(x^\prime)\|_2 \leq \ell\|x - x^\prime\|_2, \forall x, x^\prime \in \mathcal{C}$. Equivalently, $\ell\|x - x^\prime\|_2^2 \geq (x - x^\prime)^\top(\nabla h(x) - \nabla h(x^\prime))$ by Cauchy-Schwartz inequality.

\begin{proposition}
	\label{prop:U_strongly_vs_relative}
	Suppose the game $\mathcal{G}$ is, 	\begin{itemize}
		\item[(i)] $\eta$-strongly monotone,\! then $\mathcal{G}$ is $\dfrac{\eta}{\ell}$-relatively strongly monotone on $\dom(\partial h)$ with respect to any $\ell$-smooth $h$,
		\item[(ii)] $\mu$-hypo-monotone,\! then $\mathcal{G}$ is $\dfrac{\mu}{\rho}$-relatively hypo-monotone on $\dom(\partial h)$ with respect to any $\rho$-strongly convex $h$,
	\end{itemize} 
	where $h: \mathbb{R}^n \to \mathbb{R} \cup \{\infty\}$ is assumed to be differentiable, (at-least) strictly convex, with domain $\dom(h) = \Omega$.
\end{proposition} 

\cref{prop:U_strongly_vs_relative} states that, to generate a relatively strongly/hypo monotone game, one can first produce a standard strongly/hypo monotone game, then such a game will be relatively strongly monotone with respect to any $\ell$-smooth $h$ or relatively hypo-monotone with respect to any $\rho$-strongly convex $h$. The latter case when $\mathcal{G}$ is $\mu$-hypo-monotone constitute an important class of games. It was shown in \cite{Bo_LP_TAC2020} that many examples of mixed games are \textit{both} potential-free and hypo-monotone (also known as \textit{unstable games} in \cite{Bo_LP_TAC2020}). By the $1$-strong convexity of the negative entropy $h(x) = \sum_{p \in \mathcal{N}} {x^p}^\top \log(x^p),  x^p \in \Updelta^{n_p}$ \cite[p. 125]{Beck17}, \cref{prop:U_strongly_vs_relative}(ii) implies that \textit{all $\mu$-hypo-monotone mixed-extension of finite games are $\mu$-relatively hypo-monotone with respect to the negative entropy}. The same is true for $h(x) =  \frac{1}{2} \sum_{p \in \mathcal{N}} \|x^p\|_2^2, x^p \in \Updelta^{n_p}$, i.e., the Euclidean norm on the simplex.

In the following sections, we provide the rates of convergence of MD and its \textit{discounted} version \cite{Bo_LP_TechNote19} in relatively strongly monotone and relatively hypo-monotone games, respectively. While MD with \textit{time-averaged} trajectory is known to converge in null monotone games (such as the Matching Pennies game considered in \cite{Mertikopoulos16}), in this work we are interested in the rate of convergence of the \textit{actual} trajectory.

\vspace*{-0.2cm}
\subsection{Mirror Descent Dynamics}

Consider the stacked-vector representation of the mirror descent dynamics, \eqref{eqn:MD}, with rest point condition given by \eqref{eqn:MD_rest_points}:

\noindent\begin{minipage}{.55\linewidth}
\vspace{-0.3cm}
	\begin{equation}
		\label{eqn:MD_stacked}
		\begin{cases}
			\dot z & =  \gamma  U(x), \gamma  > 0,\\
			x &=  C_\epsilon(z),
		\end{cases}
		\tag{MD}
	\end{equation} 
\end{minipage}%
\begin{minipage}{.45\linewidth}
\vspace{-0.3cm}	
	\begin{equation} 
		\label{eqn:MD_rest_points} 
		\begin{cases}
			 N_\Omega(x^\star) \ni U(x^\star),&\\
			 x^\star  =  C_\epsilon(z^\star),&
		\end{cases}
	\end{equation}
\end{minipage}

\noindent where $U = (U^p)_{p \in \mathcal{N}}, C_\epsilon = (C^p_\epsilon)_{p \in \mathcal{N}}, x = (x^p)_{p \in \mathcal{N}}, z = (z^p)_{p \in \mathcal{N}}$ (similar convention used throughout). Here, we assume $x^\star$ lies in the relative interior of $\Omega$. Global convergence of the strategies generated by MD was shown for strictly monotone games \cite{mertikopoulos2017convergence}. We supplement this convergence result by showing that MD converges exponentially fast towards interior NE in $\eta$-relatively strongly monotone games.

\begin{theorem} 
	\label{thm:convergence_MD_relative_strongly_monotone}
Let $\mathcal{G}$ be $\eta$-relatively strongly monotone with respect to to $\textstyle h(x)= \sum_{p \in \mathcal{N}} \vartheta^p(x^p)$, where  $x \! = (x^p(t))_{p \in \mathcal{N}} =  C_\epsilon(z(t))$ is the solution of MD and $C_\epsilon = (C^p_\epsilon)_{p \in \mathcal{N}}$ is the mirror map induced by $\psi_\epsilon^p = \epsilon \vartheta^p$, where $\vartheta^p$ satisfies  \cref{assump:primal}. Suppose  $x^\star = ({x^p}^\star)_{p \in \mathcal{N}} \in \rinterior(\Omega)$  is the unique interior NE of $\mathcal{G}$ and let $D_{h}$ be the Bregman divergence of $h$. Then for any $\epsilon, \gamma, \eta > 0$ and any $x_0 \! = (x^p(0))_{p\in\mathcal{N}}\! =  C_\epsilon(z(0)), z(0) \in \mathbb{R}^n$,
	$x(t)$ converges to $x^\star$ with the rate, 
	\begin{equation}
		\label{eqn:MD_iterates_relative_strongly_monotone_Bregman}
		\textstyle D_h(x^\star, x) \leq   e^{-\gamma \eta\epsilon^{-1}  t} D_h(x^\star,  x_0).
	\end{equation} 	Furthermore, since $\vartheta^p$ is $\rho$-strongly convex, therefore,
	\begin{equation}
		\hspace*{-0.3cm}
		\label{eqn:MD_iterates_relative_strongly_monotone}
		\begin{split}
			\textstyle  \|x^\star - x\|^2_2 \leq 2\rho^{-1} e^{-\gamma \eta\epsilon^{-1}  t}  D_h(x^\star, x_0).
		\end{split}
	\end{equation} 
\end{theorem}

\begin{remark} Expressing the Bregman divergence of $h$ in terms of the regularizers $\vartheta^p$, we have, $\textstyle D_h(x^\star, x_0) = \sum_{p \in \mathcal{N}} D_{{\vartheta^p}}({x^p}^\star, x^p_0)$, where $D_{\vartheta^p}$ is the Bregman divergence of $\vartheta^p$. From \eqref{eqn:MD_iterates_relative_strongly_monotone_Bregman} and \eqref{eqn:MD_iterates_relative_strongly_monotone}, observe that the rate of convergence increases exponentially upon one or more of the following parameter adjustments: the learning rate factor goes up ($\gamma \uparrow$), the game becomes more strongly monotone $(\eta \uparrow)$ or the regularization goes down $(\epsilon \downarrow)$, in which the mirror map \eqref{eqn:mirror_map_argmax_char} approximates a \textit{best response} function \cite{Mertikopoulos16}.
\end{remark}

\begin{remark} \label{remark:gap} From \eqref{eqn:MD_iterates_relative_strongly_monotone_Bregman} and \eqref{eqn:MD_iterates_relative_strongly_monotone}, the distance from the initial strategy $x_0$ to the interior NE $x^\star$ also affects the rate of convergence in an intuitive way. However, different choices of the relative function $h$ will result in different upper-bounds. For example, suppose $h(x) = \sum_{p \in \mathcal{N}} \frac{1}{2} \|x^p\|_2^2$, $x^p \in \Omega^p \subset \mathbb{R}^{n_p}$, \eqref{eqn:MD_iterates_relative_strongly_monotone} can be written as, 
	\begin{equation} 
		\label{eqn:euclidean_distance_to_NE}
		\textstyle \|x^\star -  x\|_2\! \leq \sqrt{e^{-\gamma\eta\epsilon^{-1} t}  \sum_{p \in \mathcal{N}} \|{x^p}^\star - x^p(0)\|_2^2},
	\end{equation}
	whereas when $h(x) = \sum_{p \in \mathcal{N}} {x^p}^\top \log(x^p), x^p \in \Updelta^{n_p},$
	\begin{equation} 
		\label{eqn:softmax_distance_to_NE}
		\textstyle \|x^\star -  x\|_2\! \leq \sqrt{2e^{-\gamma\eta\epsilon^{-1} t}  \sum_{p \in \mathcal{N}} {{x^p}^\star}^\top \ln({x^p}^\star/x^p(0)))},
	\end{equation} 
 	where the logarithm and division are performed component-wise. 	Note that since the upper-bound for MD \eqref{eqn:MD_iterates_relative_strongly_monotone} is derived by applying $D_h(x^\star, x) \geq \frac{\rho}{2}\| x^\star - x\|_2^2$ on the LHS of \eqref{eqn:euclidean_distance_to_NE}, therefore \eqref{eqn:MD_iterates_relative_strongly_monotone} could over-estimate the distance to the NE whenever the distance between $x(t)$ and $x^\star$ is measured in terms of the Euclidean distance as opposed to the Bregman divergence. This occurs whenever the relative function $h$ is not the Euclidean norm, e.g., \eqref{eqn:softmax_distance_to_NE}. Furthermore, these upper-bounds are valid point-wise starting from $t = 0$ as long as the entire trajectory $x(t)$ remains in $\rinterior(\Omega)$ for all $t \geq 0$, which always occurs for MD with $C^p_\epsilon$ induced by steep regularizers. When the mirror map $C^p_\epsilon$ is induced by a non-steep regularizer, such as the Euclidean projection, $x(t)$ could be forced to stay along the boundary of $\Omega$ in which case the upper-bounds hold asymptotically.
\end{remark}

\begin{remark} 
	As the game becomes null monotone ($\eta \to 0$) the rate \eqref{eqn:MD_iterates_relative_strongly_monotone_Bregman} worsens to $D_h(x^\star, x) \leq D_h(x^\star, x_0)$. We note that this inequality is an equality in any (network) zero-sum games with interior equilibria (which is a subset of potential-free, monotone games), as MD is known to admit periodic orbits starting from almost every $x_0 = x(0) = C_\epsilon(z(0))$. For details, see \cite{Mertikopoulos18}. In what follows, we partially overcome this non-convergence issue through \textit{discounting} as shown in \cite{Bo_LP_TechNote19, Bo_LP_TAC2020}. 
\end{remark}

\subsection{Discounted Mirror Descent Dynamics}
Consider the \textit{discounted} MD  studied in \cite{Bo_LP_TechNote19}, which in stacked notation is  DMD with rest points  \eqref{eqn:DMD_rest_points}, 

\noindent\begin{minipage}{.55\linewidth}
	\begin{equation}
		\hspace*{-0.5cm}
		\label{eqn:DMD_stacked}
		\begin{cases}
			\dot z & =  \gamma (-z + U(x)), \gamma > 0\\
			x &=  C_\epsilon(z),
		\end{cases}
		\tag{DMD}
	\end{equation} 
\end{minipage}%
\begin{minipage}{.45\linewidth}
	\begin{equation} 
		\label{eqn:DMD_rest_points} 
		\begin{cases}
			 (N_\Omega + (C_\epsilon)^{-1}) (\overline x) \ni U(\overline x), &\\
			 \overline x =  C_\epsilon(\overline z). &
		\end{cases}
	\end{equation}
\end{minipage}

The rest point, $\overline x = C_\epsilon(U(\overline x)) =  C_\epsilon\circ U(\overline x) $, is a \textit{perturbed} NE in the following sense.

\begin{lemma}(Proposition 5, \cite{Bo_LP_TechNote19}) \label{lem:equilibria}
	Any equilibrium of the form $\overline x = C_\epsilon(U(\overline x))$,  where $C_\epsilon = (C^p_\epsilon)_{p \in \mathcal{N}}$ is the mirror map induced by $\psi_\epsilon^p = \epsilon\vartheta^p, \epsilon > 0$, satisfying \cref{assump:primal}, is the NE of the game $\mathcal{G}$ with the perturbed payoffs, 
	\begin{equation} \label{eqn:perturbed_payoff} \widetilde{\mathcal{U}}^p(x^p; x^{-p}) = \mathcal{U}^p(x^p; x^{-p})  - \epsilon\vartheta^p(x^p).\end{equation} As $\epsilon \to 0$, $\overline x \to x^\star$, where $x^\star = ({x^p}^\star)_{p \in \mathcal{N}}$ is a NE of $\mathcal{G}$. 
\end{lemma}

The existence and uniqueness of the perturbed NE was discussed in \cite[Proposition 5]{Bo_LP_TechNote19}. To summarize the remarks therein, the existence of the perturbed NE $\overline x = C_\epsilon \circ U (\overline x)$ amounts to an argument by Kakutani's fixed point theorem on the operator $C_\epsilon \circ U$. The uniqueness of $\overline x$ amounts to showing that the pseudo-gradient associated with the perturbed payoffs \eqref{eqn:perturbed_payoff} can be rendered strongly monotone due to the strong convexity assumption on $\vartheta^p$, which we show in the following proposition.

\begin{proposition}
	\label{prop:pseudo_gradient_rel_hypo_monotone} 
	Suppose $\mathcal{G}$ is $\mu$-hypo-monotone relative to $\textstyle h(x)\!=\!\sum_{p \in \mathcal{N}} \vartheta^p(x^p)$ with pseudo-gradient $U$. Let $\widetilde U = U - \Psi_\epsilon$,  $\Psi_\epsilon = (\nabla \psi_\epsilon^p)_{p \in \mathcal{N}}$, where $\psi_\epsilon^p = \epsilon \vartheta^p$ and the regularizer $\vartheta^p$ satisfies \cref{assump:primal}. Then $\widetilde U$ corresponds to the pseudo-gradient of the perturbed game $\widetilde{\mathcal{G}}$ with payoff function \eqref{eqn:perturbed_payoff} and $\widetilde{\mathcal{G}}$ is $(\epsilon - \mu)$-strongly monotone relative to $\textstyle h$ whenever $\epsilon> \mu$.
\end{proposition}

Assuming  \cref{prop:pseudo_gradient_rel_hypo_monotone} holds, then $\widetilde{\mathcal{G}}$ is relatively strongly monotone as long as $\epsilon > \mu$, and the convergence rate for DMD follows that of MD in  \cref{thm:convergence_MD_relative_strongly_monotone}. The convergence of the strategies generated by DMD was shown in \cite{Bo_LP_TechNote19}; we supplement the results therein by providing the rate of convergence in $\mu$-relatively hypo-monotone and null monotone ($\mu = 0$) games.

\begin{corollary} 
	\label{thm:convergence_DMD_relative_hypo_monotone}
	Let $\mathcal{G}$ be $\mu$-relatively hypo-monotone with respect to  $\textstyle h(x) = \sum_{p \in \mathcal{N}} \!\vartheta^p(x^p)$, where  $x = (x^p(t))_{p \in \mathcal{N}} =  C_\epsilon(z(t))$ is the solution of DMD and $C_\epsilon = (C^p_\epsilon)_{p \in \mathcal{N}}$ is the mirror map induced by $\psi_\epsilon^p = \epsilon \vartheta^p$, where $\vartheta^p$ satisfies \cref{assump:primal}.  Suppose  $\overline x = (\overline x^p)_{p \in \mathcal{N}} \in \rinterior(\Omega) $  is the unique perturbed interior NE of $\mathcal{G}$ and let $D_{h}$ be the Bregman divergence of $h$. Then for any $\gamma,   \rho > 0, \epsilon > \mu$ and any $x_0 = x(0)\! = (x^p(0))_{p\in\mathcal{N}}\! =  C_\epsilon(z(0))$,
	$x(t)$ converges to $\overline x$ with the rate, 
	\begin{equation}
		\label{eqn:DMD_iterates_rel_hypomonotone_Bregman}
		\textstyle  D_h(\overline x, x) \leq   e^{-\gamma(\epsilon - \mu)\epsilon^{-1} t} D_h(\overline x, x_0).
	\end{equation} 	
	Furthermore, since $\vartheta^p$ is $\rho$-strongly convex, therefore,
	\begin{equation}
		\label{eqn:DMD_iterates_rel_hypomonotone}
		\hspace*{-0.3cm}
		\begin{split}
			\textstyle  \|\overline x -  x\|^2_2 \leq 2\rho^{-1} e^{-\gamma(\epsilon - \mu)\epsilon^{-1} t} D_h(\overline x, x_0).
		\end{split}
	\end{equation} 
	For $\mu = 0$ ($\mathcal{G}$ is null monotone), \eqref{eqn:DMD_iterates_rel_hypomonotone} implies,
	\begin{equation}
		\label{eqn:DMD_iterates_null_monotone}
		\hspace*{-0.3cm}
		\begin{split}
			\textstyle  \|\overline x -  x\|^2_2 \leq 2\rho^{-1} e^{-\gamma t} D_h(\overline x, x_0).
		\end{split}
	\end{equation} 
\end{corollary}

\begin{remark}
	We note that the condition, $\epsilon > \mu$, coincides with the known convergence condition of DMD in hypo-monotone games in \cite{Bo_LP_TAC2020}. This condition appears as $\epsilon > \mu \rho^{-1}$ in \cite{Bo_LP_TechNote19} due to $\rho$ arising from the standard (non-relative) notion of strong convexity used in the proofs therein. Hence whenever DMD converges in a $\mu$-hypo-monotone game (which by \cref{prop:U_strongly_vs_relative} is relatively hypo-monotone), it converges with the rate according to \cref{thm:convergence_DMD_relative_hypo_monotone}. 
	In addition, since any $\eta$-relatively strongly monotone function is $\mu$-hypo-monotone with $\mu = -\eta$, therefore DMD converges in the relatively strongly monotone regime with rate \eqref{eqn:DMD_iterates_rel_hypomonotone} where $\mu$ is replaced with $-\eta$, which implies an improved rate. Hence, compared with \eqref{eqn:MD_iterates_relative_strongly_monotone}, DMD tends to converge faster to nearly (or exactly) the same NE as MD for small $\epsilon$. 
\end{remark}

\section{Rate of Convergence in Games with Relatively Strongly Concave Potential}
\label{sec:relatively_strongly_concave} 

In this section, we consider the special case whereby the game admits a potential function. These games form a subset of the monotone games that we have discussed so far. Recall from \cref{sec:review_concepts}, a concave potential game satisfies the relationship $\nabla P = U$ where $P$ is some scalar-valued function, hence all of the definitions associated with monotone games can be rephrased in terms of $P$. For simplicity, in what follows, we assume that $\Omega$ has a non-empty interior.\footnote{In the case for which $\interior(\Omega) = \varnothing$, one could construct a \textit{full potential} game. For details, see \cite{Sandholm10}.}

 Following the convention from optimization, these definitions are usually stated as follows \cite{Beck17}:
\begin{definition}
	\label{def:potential_concave} 
	The potential function $P: \Omega \to \mathbb{R}$ is,
	\begin{itemize}
		\item[(i)] concave if $P(x)\! \leq \!P(x^\prime) + \!\nabla P(x^\prime)^\top (x-x^\prime), \forall x \in \Omega,  x^\prime\! \in \interior(\Omega),$
		\item[(ii)] $\eta$-strongly concave if,  $P(x) \leq P(x^\prime) + \nabla P(x^\prime)^\top (x-x^\prime) - \frac{\eta}{2}\|x-x^\prime\|_2^2, \forall x \in \Omega,  x^\prime\! \in \interior(\Omega),$ for some $\eta > 0$,
		\item[(iii)] $\mu$-weakly concave if, $P(x) \leq P(x^\prime) + \nabla P(x^\prime)^\top (x-x^\prime) + \frac{\mu}{2}\|x- x^\prime\|_2^2, \forall x \in \Omega,  x^\prime\! \in \interior(\Omega),$ for some $\mu > 0$.
	\end{itemize}
\end{definition}
In contrast to relative monotonicity, \textit{relative concavity} have been previously investigated in the optimization context \cite{Lu_Relative_SC, Xu18}. We provide a slightly extended version of relative strong concavity as compared to \cite{Lu_Relative_SC}. 
\begin{definition}
	\label{def:potential_relatively_concave}
	Suppose $h: \mathbb{R}^n \to \mathbb{R} \cup \{\infty\}$ is any differentiable convex function with domain $\dom(h) = \Omega$. The potential function $P$ is $\eta$-strongly concave relative to $h$, if for all $x \in \dom(\partial h), x^\prime \in \dom(\partial h)$, for some $\eta > 0$,
	\begin{equation} 		\label{eqn:potential_SC_Bregman}  P(x) \leq P(x^\prime) + \nabla P(x^\prime)^\top (x-x^\prime) - \eta D_{h}(x, x^\prime).
	\end{equation} 
\end{definition}
\begin{remark} 
Analogously, $P$ is $\mu$-weakly concave relative to $h$ for some $\mu > 0$ if for all $x \in  \dom(\partial h), x^\prime \in \dom(\partial h)$, $P(x) \leq P(x^\prime)\! +\! \nabla P(x^\prime)^\top\! (x-x^\prime)\! +\! \mu D_{h}(x, x^\prime).$ When $h$ is $\frac{1}{2} \|x\|^2_2$, \cref{def:potential_relatively_concave} implies \cref{def:potential_concave}(ii).
\end{remark} 
\begin{remark} \label{remark:1}  Note that  $P$ is $\eta$-strongly concave relative  to $h$ (or equivalently, \textit{$\eta$-relatively strongly concave}) if $P + \eta h$ is concave. Moreover, any $\eta$-strongly concave potential is $\eta \ell^{-1}$-relatively strongly concave with respect to any $\ell$-smooth $h$. Hence, to generate a game with a potential function that is strongly concave relative with respect to some function $h$, one can first generate a potential function that is standard strongly concave, then this game will be relatively strongly concave with respect to any relative function $h$ that is $\ell$-smooth. In the same vein, one can first generate a potential function that is standard weakly concave, then this potential will be relatively weakly concave with respect to any relative function $h$ that is $\rho$-strongly convex.
\end{remark}

\subsection{Actor-Critic Dynamics} Since all the potential games that we have discussed so far are also monotone games, hence our results in the previous section apply to MD as well regardless of whether the game possesses a potential function or not. Hereby we exclusively focus our attention on AC, which is only known to converge in potential games. Recall that the AC dynamics is given as \eqref{eqn:AC},
\begin{equation}
	\label{eqn:AC_stacked}
	\dot z  = \gamma  U(x), \gamma > 0, \quad 
	\dot x =  r (C_\epsilon(z) - x), r > 0,
	\tag{AC}
\end{equation}
We note that the rest points of AC are the same ones as those of MD \eqref{eqn:MD_rest_points}.  
 
\begin{theorem} \label{thm:convergence_AC_relative_sc}
	Let $\mathcal{G}$ be a potential game with $P$ $\eta$-strongly concave relative with respect to $\textstyle h(x)\! =\! \!\sum_{p \in  \mathcal{N}} \!\vartheta^p(x^p)$, where  $x \!=\! (x^p(t))_{p \in \mathcal{N}} \!\!=\!\! C_\epsilon(z(t))$ is the solution of AC,  $C_\epsilon \!=\! (C^p_\epsilon)_{p \in \mathcal{N}}$ is the mirror map induced by $\psi_\epsilon^p \!=\! \epsilon \vartheta^p$, and $\vartheta^p$ satisfies \cref{assump:primal}. Suppose $x^\star \!\!=\! ({x^p}^\star)_{p \in \mathcal{N}} \in \interior(\Omega)$ is the unique interior NE of $\mathcal{G}$ and let $D_{h}$ be the Bregman divergence of $h$. Then for any $r, \gamma,  \eta \! >\! 0, \epsilon  \!>\! \eta\gamma/r$ and any $x_0\!\! =\! \!(x^p(0))_{p\in\mathcal{N}}\! \!=  \!C_\epsilon(z(0))$, $z(0) \in \mathbb{R}^n$, \vspace{-0.2cm}
	\begin{equation}
		\hspace*{-0.2cm}
		\label{eqn:AC_potential_relative_sc}
		\textstyle P(x^\star) -  P(x)  \leq  e^{-\gamma \eta \epsilon^{-1}  t} (P(x^\star) - P(x_0) +r\epsilon\gamma^{-1}D_h(x^\star,x_0)),
	\end{equation}
	and $x(t)$ converges to $x^\star$ with the rate, \vspace{-0.2cm}
	\begin{equation}
		\label{eqn:thm_AC_iterates_relative_sc_bregman}
		\textstyle D_h(x, x^\star) \leq   \eta^{-1} e^{-\gamma \eta \epsilon^{-1}  t} (  P(x^\star)- P(x_0)  +r\epsilon\gamma^{-1}D_h(x^\star, x_0)).
	\end{equation} Furthermore, since $\vartheta^p$ is $\rho$-strongly convex, therefore,
	\begin{equation}
		\hspace*{-0.1cm}
		\label{eqn:thm_AC_iterates_relative_sc_euclidean}
		\textstyle \|x^\star - x\|^2_2 \leq  2(\rho \eta)^{-1} e^{-\gamma \eta \epsilon^{-1} t} (P(x^\star) - P(x_0)  + r\epsilon\gamma^{-1} D_h(x^\star, x_0)).
	\end{equation} 
\end{theorem} 

\section{Case Studies} 
\label{sec:case_studies} 
We present several case studies for monotone games, whereby the games do not admit potentials, and demonstrate the validity of these upper-bounds. We note that since several examples involving the rates of AC in potential games were considered in \cite{Bo_LP_CDC2020}, thus we do not consider them here. In the following examples, we begin by considering the strongly monotone case where both MD and DMD converge. Next, we consider a null monotone game, followed by a hypo-monotone game, where MD ceases to converge but DMD still converges, see \cite{Bo_LP_TechNote19, Bo_LP_TAC2020}.

\begin{example}(\textbf{Adversarial Attack On An Entire Dataset})
	\label{example:adversarial_perturbation} Consider a single dataset $\mathpzc{D} = \{(a_n,b_n)\}_{n = 1}^N$, with examples $a_n \in \mathbb{R}^d$ and associated labels $b_n \in \mathbb{Z}$, and a \textit{trained} model $\mathpzc{M}: \mathbb{R}^d \to \mathbb{Z}$, $\mathpzc{M}(a_n) = b_n, \forall n$. We assume that an attacker wishes to produce a single perturbation $\iota \in \mathcal{I} \subseteq \mathbb{R}^n$, such that when it is added to every examples, each of the new examples $\widehat{a}_n \coloneqq a_n + \iota$ \textit{potentially} causes $\mathpzc{M}$ to misclassify, all the while the difference $\widehat{a}$ and $a_n$ remains small, i.e., the attacker also wishes for the perturbation to be imperceptible. This is a weaker version of an ``universal perturbation", where every perturbed example causes the model to misclassify \cite{Zhang19}. 
	
	An approach that induces a convex-concave saddle point problem is as follows:  first, construct a set of new labels (``targets")  $\widehat b_n$ which may be derived from the true labels $b_n$. Next, minimize the distance between the prediction on $\widehat{a}_n$ with its associated $\widehat {b}_n$ by calculating $\iota$ against the worst-case convex combination of the convex loss functions. This routine can be formulated as,
	\begin{equation}
		\textstyle  \underset{\iota \in \mathcal{I}}{\text{min}} \thinspace \underset{\mathpzc{p} \in\Updelta^N}{\text{max}} \thinspace  \mathpzc{p}^\top \mathpzc{F}(\iota; a_n, \widehat{b}_n, w) - \mathpzc{R}(\mathpzc{p}),
		\label{eqn:universal_adversarial_example_optimization}
	\end{equation} 	
	where $\mathpzc{R}$ is a convex regularizer, $\mathpzc{F} = (\mathpzc{L}^n)_{n=1}^N$ is a stacked-vector whereby each $\mathpzc{L}^n$ is a (per-sample) loss function  that models the distance between the prediction on the $n^\text{th}$ \textit{perturbed} example and the $n^\text{th}$ \textit{perturbed} label, $\mathpzc{p} \in \Updelta^{N}$ describes the worst-case convex combination of the loss functions, $w \in \mathbb{R}^{d+1}$ is the weight of the  model $\mathpzc{M}$. 
	
	Consider a trained logistic regression model, $\mathpzc{M}(w) = \varphi(w^\top a_n)$, where $\varphi: \mathbb{R}\to (0,1)$ is the logistic function, $\varphi(x) = \exp(x)(1+\exp(x))^{-1}$ (for background, see \cite[p. 246]{Murphy}), hence each of the convex loss function in the stacked-vector $\mathpzc{F}$ is of the form $\mathpzc{L}^n(\iota; a_n, \widehat b_n, w) = \log(1 +\exp(-\widehat b_n (w_0 + w_1 (a_n+\iota))))$, $w = (w_0, w_1) \in \mathbb{R}^2$. Suppose $a_n \in \mathbb{R}$, and $b_n \in \{-1, +1\}$, the new/adversarial targets $\widehat b_n \in \{+1, -1\}$ is obtained by flipping each $b_n$. Let $ x^1  = \iota \in [-1, 1]$, $x^2 = \mathpzc{p}  \in \Updelta^N$, and $\mathpzc{R}(x^2) = \frac{\mathpzc{r}}{2}\|x^2 - \mathbf{1} /N\|_2^2$, $\mathpzc{r} > 0$,  then \eqref{eqn:universal_adversarial_example_optimization} is equal to,
	\begin{equation}
		\textstyle \underset{x^1  \in [-1,1]}{\text{min}} \thinspace \underset{x^2 \in \Updelta^N}{\text{max}} \thinspace {x^2}^\top \mathpzc{F}(x^1; a_n, \widehat{b}_n, w) - \frac{\mathpzc{r}}{2}\|x^2 - \mathbf{1} /N\|_2^2,
	\end{equation}  
	which is equivalent to a two-player concave zero-sum game with payoff functions, 
	\begin{equation} 
		\textstyle  \mathcal{U}^1(x^1; x^2) = -{x^2}^\top \mathpzc{F}(x^1; a_n, \widehat{b}_n, w) + \frac{\mathpzc{r}}{2}\|x^2 - \mathbf{1} /N\|_2^2,
	\end{equation} 
	and $\mathcal{U}^2(x^1; x^2) = -\mathcal{U}^1(x^1; x^2).$ The pseudo-gradient of the game is,
	\begin{equation}
		\textstyle 	U(x) = \begin{bmatrix} \sum_{n = 1}^N  x^2_n \widehat b_n w_1  \varphi(-\widehat{b}_n (w_0 + w_1 (a_n + x^1)))   \\ \mathpzc{F}(x^1; a_n, \hat{b}_n, w) - \mathpzc{r}(x^2 - \mathbf{1} /N)\end{bmatrix},
	\end{equation}
	The Jacobian of $U$ is,
	\begin{equation}
		\mathbf{J}_{U}(x) = \begin{bmatrix}  
			\star & \mathpzc{m}^1 & \ldots  & \mathpzc{m}^N \\ 
			-\mathpzc{m}^1 & -\mathpzc{r} & \ldots & 0 \\ 
			\vdots & \vdots &  \ddots & \vdots \\ 
			-\mathpzc{m}^N & 0  & \ldots & -\mathpzc{r} \end{bmatrix},
	\end{equation}
	where each $\mathpzc{m}^n \!= \widehat b_n w_1  \varphi(-\widehat{b}_n w_0 -\widehat{b}_n w_1 (a_n + x^1))$, $\star = -\sum_{n = 1}^N x^2_n (\widehat{b}_n w_1)^2 \varphi(-\widehat{b}_n w_0  -\widehat{b}_n w_1 (a_n  + x^1))(1- \varphi(-\widehat{b}_n w_0 -\widehat{b}_n w_1 (a_n + x^1))$. Which means $\mathcal{G}$ is $\eta$-strongly monotone for $\eta =  |\max(\star, -\mathpzc{r})|, \forall x^1 \in [-1,1], x^2 \in \Updelta^N$. By \cref{prop:U_strongly_vs_relative}(i), $\mathcal{G}$ is $\mu$-relatively strongly monotone with respect to $h(x) = \frac{1}{2}\|x\|_2^2, \forall x \in [-1,+1] \times \Updelta^N$.
	
	\begin{figure}[htp!]
		\vspace{-0.3cm}
		\centering
		\includegraphics[draft = false, width= 0.5\linewidth]{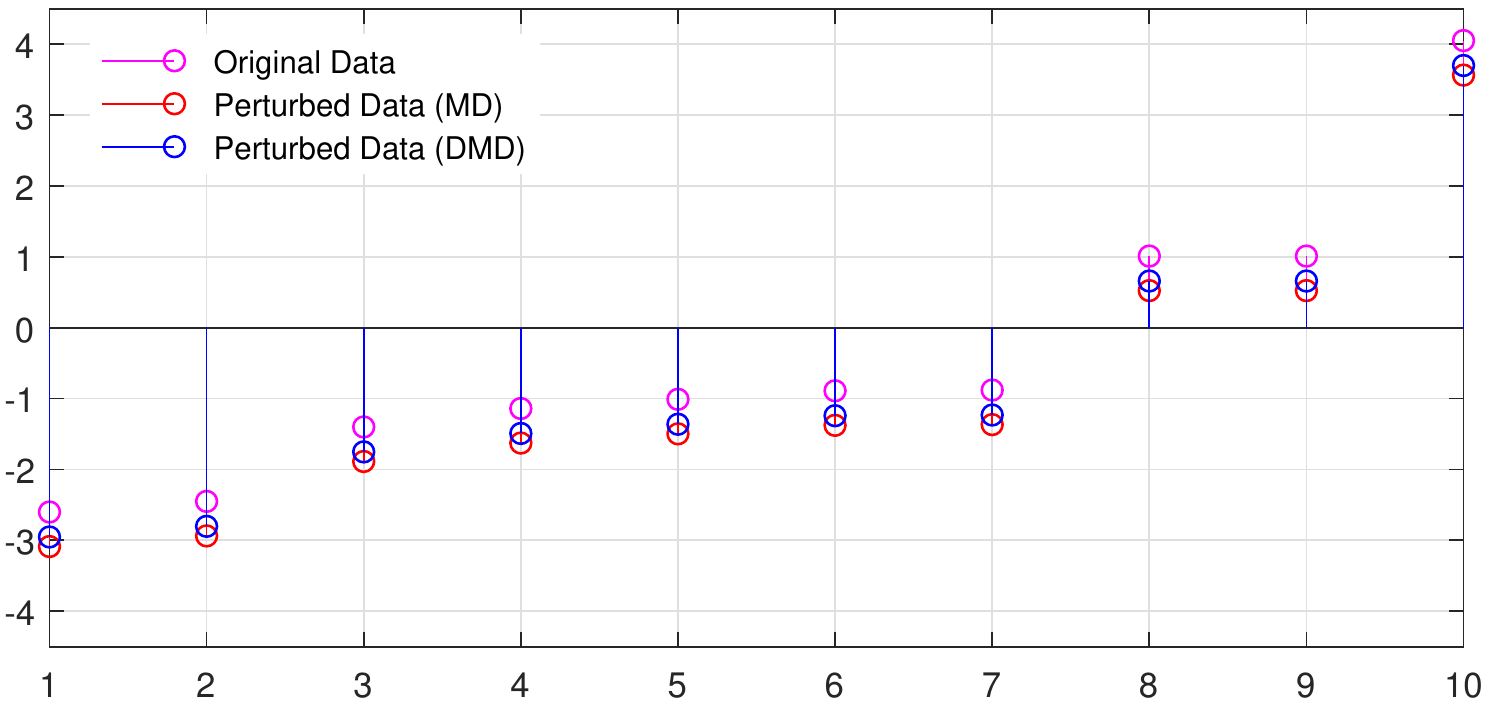}
		\caption{\textit{Original data versus perturbed data in the adversarial attack example.}}
		\label{Fig:Adversarial_Attack_Data}
		\vspace{-0.1cm}
	\end{figure}
	We consider a linearly separable dataset $\mathpzc{D}$ with $N = 10$ examples generated according a Gaussian distribution $\mathpzc{N}(0, 4)$, sorted from smallest to largest, with $-1$ labels generated for examples with larger magnitudes, and $+1$ otherwise (see \cref{Fig:Adversarial_Attack_Data}), where the trained classifier $\mathpzc{M}$ has weights $w = (0.8484, 0.8947)$. We set $\mathpzc{r} = 10, \gamma = 1, \epsilon = 0.1, x^p_0 = C_\epsilon(z^p(0)), z^1(0) = 1, z^2(0) = \mathbf{1}/10$,   MD converge to approximately ${x^1}^\star = -0.487, {x^2}^\star = (0.11,0.11,0.09,0.09,0.09,0.08,0.08,0.10,0.10,0.15)$ while DMD converge to $\overline x^1 = -0.352$ and $\overline x^2 = (0.23,0.12,0.10,0.03,0.02,0.01,0.01,0.09,0.14,0.27)$ when rounded. 
	In both cases, the resulting perturbations $\iota$ manages to fool $\mathpzc{M}$ simultaneously on two of the examples ($6$th and $7$th example in \cref{Fig:Adversarial_Attack_Data}). The perturbation calculated by DMD is smaller, thus better meets the requirement that the change should be imperceptible. The trajectories of MD (dotted) and DMD (solid) are shown in \cref{Fig:Adversarial_Attack_Convergence}. The location of the true NE of the game is indicated by green stars.

	The strong monotonicity parameter is  the max amongst	$\{-(\widehat{b}_n w_1)^2 \varphi(-\widehat{b}_n w_0  -\widehat{b}_n w_1 (a_n  + x^1))(1- \varphi(-\widehat{b}_n w_0 -\widehat{b}_n w_1 (a_n + x^1))\} \cup \{-\mathpzc{r}\}$ over $x^1 \in [-1, 1]$. Using grid-search we find that $\eta =  0.0037$, which occurs at $x^1 = 1$ and the largest $(a_n,b_n)$ pair. The comparison of the rate between MD (red) and DMD (blue) is shown in \cref{Fig:Adversarial_Attack_MD_vs_DMD_rate}. While the upper-bound initially under-estimates the true trajectory due to the projection operator (see \cref{remark:gap}), it provides a reasonable asymptotic description. 
	
	\begin{figure}[ht]
		\centering
		\begin{minipage}[ht]{0.45\columnwidth}
			\hspace{-0.2cm}	\centerline{\includegraphics[draft = false, width=1.1\linewidth]{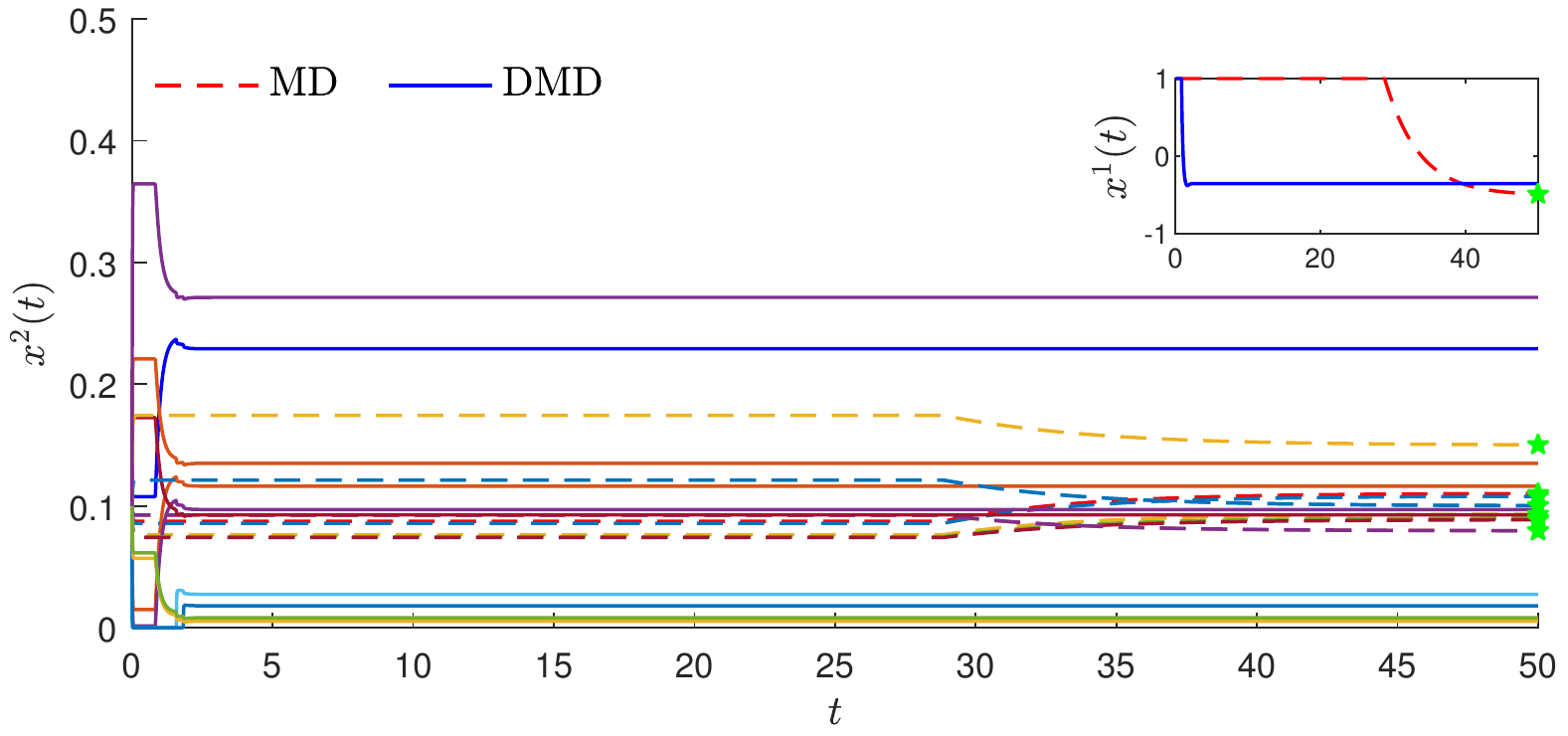}}
			\caption{Convergence of trajectories in the targeted adversarial attack example. MD is shown as dotted curves. DMD is shown as solid curves.}  
			\label{Fig:Adversarial_Attack_Convergence}
		\end{minipage}
		\hspace{0.2cm}
		\begin{minipage}[ht]{0.45\columnwidth}
			\vspace{-0.1cm}	\centerline{\includegraphics[draft = false, width=1.1\linewidth]{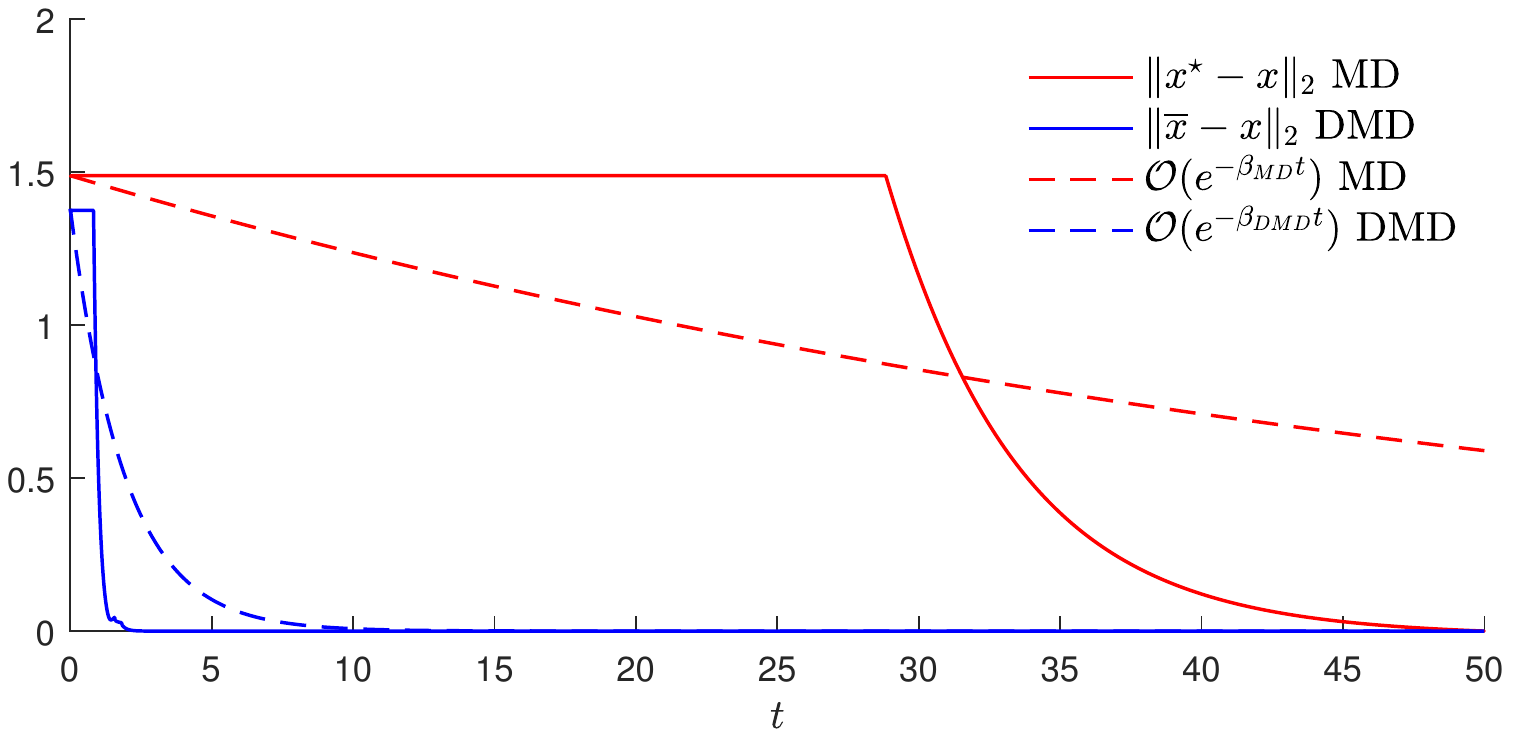}}
			\caption{Comparison between the rate of convergence of MD and DMD for the adversarial attack example. Note $\beta_\text{MD} = 0.0037, \beta_\text{DMD} = 1+\beta_\text{MD}$} 
			\label{Fig:Adversarial_Attack_MD_vs_DMD_rate}
		\end{minipage}
	\vspace{-0.5cm}
	\end{figure}
\end{example}

The next two examples are in the setting of finite (mixed) games. Recall that a finite game is the triple $\mathcal{G} = (\mathcal{N}, (\mathcal{A}^p)_{p \in \mathcal{N}} , (\mathcal{U}^p)_{p \in \mathcal{N}})$  where each player $p \in \mathcal{N}$ has a finite set of  strategies $\mathcal{A}^p = \{1, \ldots, n^p\}, n^p \geq 1$ and a payoff function $\mathcal{U}^p: \mathcal{A} = \prod_{p \in \mathcal{N}} \mathcal{A}^p \to \mathbb{R}$, where $\mathcal{U}^p(i) = \mathcal{U}^p(i^1, \ldots, i^p, \ldots,  i^N)$ denotes the payoff for the $p$th player when each player chooses a strategy $i^p \in \mathcal{A}^p$. Then the mixed-extension of the game (\textit{mixed game}) is also denoted by $\mathcal{G} = (\mathcal{N}, (\Updelta^p)_{p \in \mathcal{N}} , (\mathcal{U}^p)_{p \in \mathcal{N}})$, where $\Updelta^p \coloneqq \Updelta^{n_p} =  \{x^p \in \mathbb{R}^{n_p}_{\geq 0}|\|x^p\|_1 = 1\}$ is the set of mixed strategies for player $p$, and each player's expected payoff is $\mathcal{U}^p: \Updelta =  \prod_{p \in \mathcal{N}} \Updelta^p\to \mathbb{R}$, 	$\mathcal{U}^p(x) = \sum_{i \in \mathcal{A}} \mathcal{U}^p(i) \prod_{p \in \mathcal{N}} x^p_{i^p} = \sum_{i \in \mathcal{A}^p} {x_i^p} U^p_i(x) = {x^p}^\top U^p(x)$, where $U^p_i(x) = \mathcal{U}^p(i; x^{-p})$ and $U^p = (U^p_i)_{i \in \mathcal{A}^p}$ is referred to as the player $p$'s payoff vector. The (overall) payoff vector $U = (U^p)_{p \in \mathcal{N}}$ is equivalent to the the pseudo-gradient of the mixed-game.

We simulate each of the following examples using DMD with two mirror maps: the Euclidean projection onto the simplex and the softmax. We then compare their distances to the NE along with their theoretical upper-bounds. Since MD do not converge in these examples (see \cite{Bo_LP_TechNote19, Bo_LP_TAC2020, Mertikopoulos18}) therefore we do not consider it. 

\begin{example}[\textbf{Three-Players Network Zero-Sum Game}] 	Consider a network represented by a finite, fully connected, undirected graph $\mathpzc{G} = (\mathpzc{V}, \mathpzc{E})$ where $\mathpzc{V}$ is the set of vertices  (players) and $\mathpzc{E} \subset \mathpzc{V} \times \mathpzc{V}$ is the set of edges which models their interactions. 
	Given two vertices (players)
	$p,q \in \mathpzc{V}$,  
	we assume that there is a zero-sum game on the edge 
	$(p,q)$ given by the payoff matrices $(\mathpzc{A}^{p,q}, \mathpzc{A}^{q,p})$, whereby $\mathpzc{A}^{q,p} = -{\mathpzc{A}^{p,q}}$. Assume that $N=3$ players, where each pair plays a Matching Pennies (MP) game, \vspace{-0.2cm}\begin{equation}
		\label{eqn:network_mp_game}
		\mathpzc{A}(k) = \begin{bmatrix*}[r] +k & -k \\ -k & +k \end{bmatrix*},
	\end{equation}
	with the NE  $x^\star = ({x^p}^\star)_{p \in \mathcal{N}}, {x^p}^\star = (1/2, 1/2)$. The perturbed NE $\overline x$ coincides with the true NE in this game \cite{Bo_LP_TAC2020}. Let the payoff matrices for each edge be given as,
	\[ \begin{array}{lll}%
		\mathpzc{A}^{1,2} = \mathpzc{A}(1) & \mathpzc{A}^{1,3} = \mathpzc{A}(2) & \mathpzc{A}^{2,3} = \mathpzc{A}(3)\\   
		\mathpzc{A}^{2,1} = -{\mathpzc{A}^{1,2}}  &  \mathpzc{A}^{3,1} = -{\mathpzc{A}^{1,3}}   &  \mathpzc{A}^{3,2} = -{\mathpzc{A}^{2,3}}
	\end{array}\]
	Since each pair-wise interaction between players is a zero-sum game, the pseudo-gradient (payoff vector) of the overall player set is given by, \vspace{-0.2cm}
	$$U(x) =  \begin{bmatrix} U^1(x) \\ U^2(x)\\U^3(x) \end{bmatrix} =  \begin{bmatrix} 0 &  \mathpzc{A}^{1,2} & \mathpzc{A}^{1,3}\\ -{\mathpzc{A}^{1,2}}^\top  & 0 & \mathpzc{A}^{2,3} \\ -{\mathpzc{A}^{1,3}}^\top  & -{\mathpzc{A}^{2,3}}^\top  & 0 \end{bmatrix}\begin{bmatrix}\vphantom{-{\mathpzc{A}^{1,2}}^\top } x^1 \\ x^2 \\ x^3  \vphantom{-{\mathpzc{A}^{1,2}}^\top }\end{bmatrix}
	$$
	or $U(x) = \Phi x$, where $\Phi+\Phi^\top=0$, i.e., the game is null monotone game ($\mu = 0$). 
	
	We simulate DMD with player parameters set to be $\gamma = 1, \epsilon = 1$, and initial condition $x^p_0 =  C_\epsilon(z^p(0)), z^p(0) = (1,2)$. We plot the distances to the NE  along with their upper-bounds in \cref{Fig:Network_MP_Convergence_MD}. Since the entire solution for either DMD with softmax or Euclidean projection stays in the inteiror of the simplex, therefore by \cref{remark:gap}, these upper-bounds are valid for all $t \geq 0$. 
\end{example}
\begin{figure}[ht]
	\centering
	\begin{minipage}[ht]{0.45\columnwidth}
		\hspace{-0.2cm}	\centerline{\includegraphics[draft = false, width=1.1\linewidth]{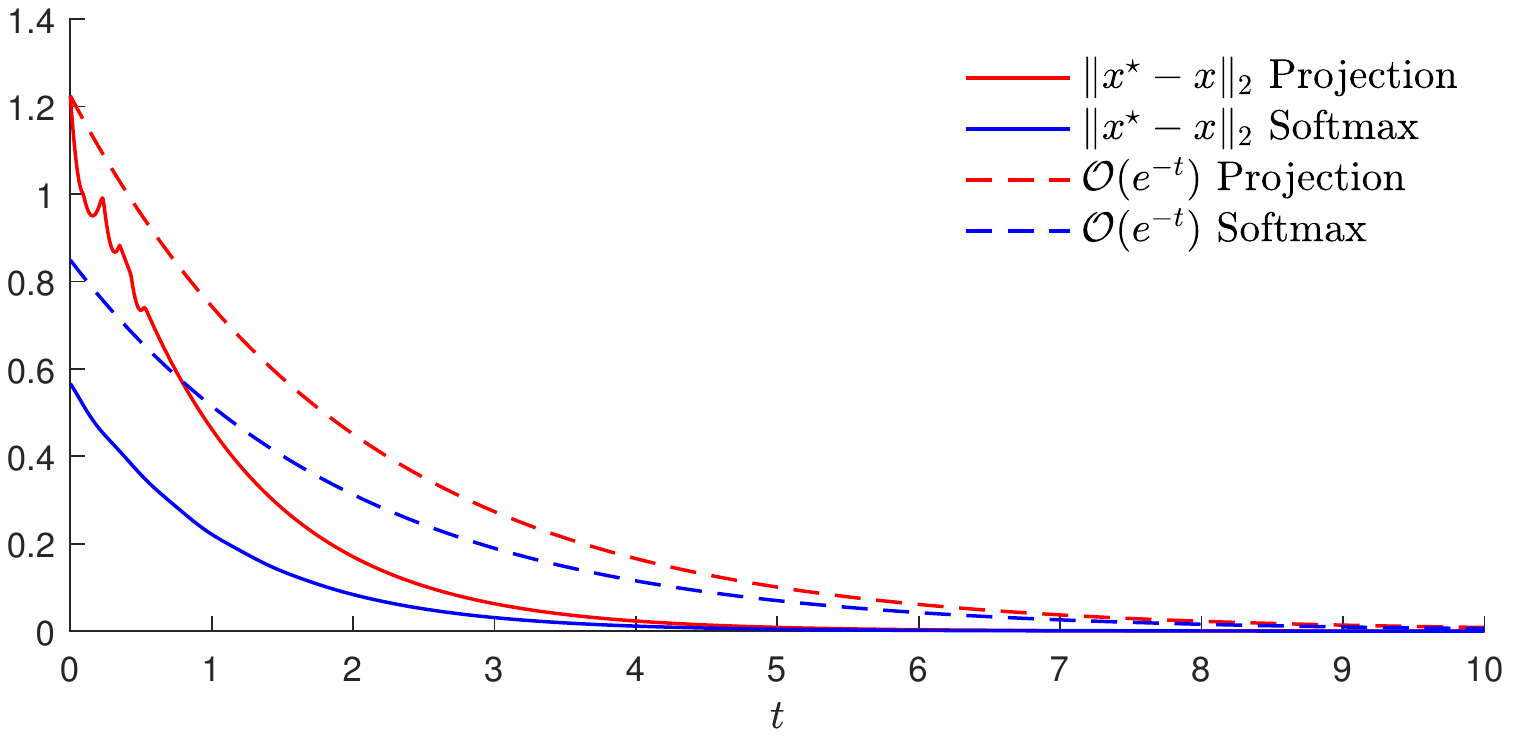}}
		\caption{Comparison between the rate of convergence of DMD with Euclidean projection versus with softmax in Network Zero-Sum Matching Pennies game.}  
		\label{Fig:Network_MP_Convergence_MD}
	\end{minipage}
	\hspace{0.2cm}
	\begin{minipage}[ht]{0.45\columnwidth}
		\vspace{-0.1cm}	\centerline{\includegraphics[draft = false,, width=1.1\linewidth]{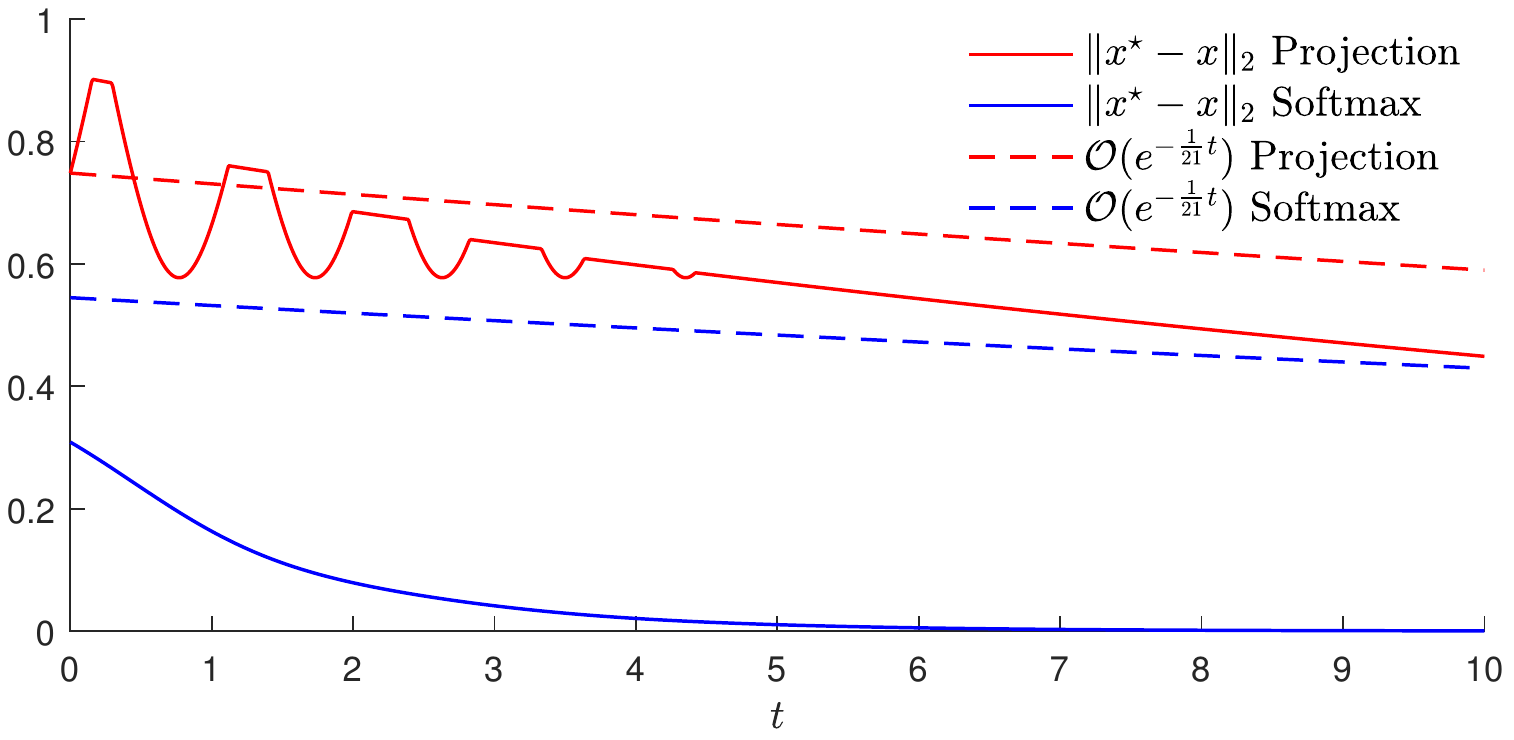}}
		\caption{Comparison between the rate of convergence of DMD with Euclidean projection versus with softmax in the two-players RPS game example.} 
		\label{Fig:RPS_convergence_rates}
	\end{minipage}
	\vspace{-0.5cm}
\end{figure}
\begin{example}[\textbf{Two-Player Rock-Paper-Scissors (RPS)}]\!\!	Consider a two-player RPS game with $\mathpzc{A}$ and $\mathpzc{B}$ being the payoff matrices for player $1$ and $2$,
	\begin{equation} 
		\mathpzc{A} = \begin{bmatrix} 0 & -\mathpzc{l} &\mathpzc{w} \\ \mathpzc{w} & 0 & -\mathpzc{l} \\ -\mathpzc{l} & \mathpzc{w} & 0 \end{bmatrix}, \quad \mathpzc{B} = \mathpzc{A}^\top,
	\end{equation} 
	where $\mathpzc{l}, \mathpzc{w} > 0$ are the values associated with a loss or a win. For this game, the pseudo-gradient (or payoff vector) is,
	\begin{equation}
		\label{eqn:shapley}
		U(x) = \begin{bmatrix} 0 & \mathpzc{A} \\ \mathpzc{B}^\top & 0 \end{bmatrix} \begin{bmatrix} x^1\\ x^2 \end{bmatrix} = \begin{bmatrix}  \mathpzc{A}  x^1 \\   \mathpzc{A} x^2 \end{bmatrix}.
	\end{equation}
	Following the arguments in \cite{Bo_LP_TAC2020}, it can be shown that $\mathcal{G}$ is $\mu$-hypo-monotone with $\mu = \frac{1}{2}|\mathpzc{l} - \mathpzc{w}|$, for all $\mathpzc{l} \neq \mathpzc{w}$, and null monotone for $\mathpzc{l} = \mathpzc{w}$, hence by \cref{prop:U_strongly_vs_relative}(ii), $\mathcal{G}$ is $\mu$-relatively hypo-monotone with respect to $h(x) = \sum_{p \in \mathcal{N}} {x^p}^\top \log(x^p)$ or $h(x) = \frac{1}{2} \sum_{p \in \mathcal{N}} \|x^p\|_2^2$. The NE of this game is $x^\star = ({x^p}^\star)_{p \in \mathcal{N}}, {x^p}^\star \!= \!(1/3, 1/3, 1/3)$, which coincides with $\overline x$ regardless of the value of $\epsilon$ \cite{Bo_LP_TAC2020}.
	
	From \cref{thm:convergence_DMD_relative_hypo_monotone}, DMD converges for any $\epsilon >  \frac{1}{2}| \mathpzc{l} -  \mathpzc{w}|$ with the rate \eqref{eqn:DMD_iterates_rel_hypomonotone}. We simulate DMD for an example with $\mathpzc{w} = 1, \mathpzc{l} = 5$, i.e., $\mu = 2$, and set players' parameters to be $\gamma = 1, \epsilon = 2.1$, with initial condition $x^p_0 = C_\epsilon(z^p(0)), z^p(0) = (1,2,3)$.  We plot the distances to the NE along with their upper-bounds in \cref{Fig:RPS_convergence_rates}.  

	Our result shows a close match between the distances along with their upper-bounds and conclusively shows that DMD with softmax is faster than DMD with Euclidean projection for this game. Furthermore, \cref{Fig:RPS_convergence_rates}, we see that that despite the extremely slow convergence of  DMD with Euclidean projection (e.g., does not converge even for $t = 100$), the exponentially decaying upper-bound is still able to accurately capture its rate of convergence (dotted, red). The gap between the dotted and the solid blue line in \cref{Fig:RPS_convergence_rates} can be made closer by using the Bregman divergence \eqref{eqn:DMD_iterates_rel_hypomonotone_Bregman} instead of the Euclidean distance \eqref{eqn:DMD_iterates_rel_hypomonotone} (see \cref{remark:gap}). 
\end{example}

\section{Conclusions}
\label{sec:conclusions}

In this paper, we have provided the rate of convergence for two main families of continuous-time dual-space game dynamics in $N$-player continuous monotone games, with or without potential. We have shown MD and DMD converge with exponential rates as long as its mirror map $C_\epsilon$ is generated with a regularizer that is matched to the geometry of the game, characterized through a relative function. Similarly, AC was also shown to exhibit exponential convergence in games with relatively strongly concave potential. Through this work, we clearly demonstrate the importance of geometry when analyzing the rates of dual-space dynamics.  

There are several open questions from our analysis. First, \textit{our results do not capture the rate of convergence towards boundary NEs}. However, in practice we have found that these bounds are still quite predictive. It is worth noting that many of the regularizers (such as generalized entropy) are not $\ell$-smooth over their entire domains \cite{Bo_LP_TechNote19}. Yet, the MD associated with these type of regularizers have been empirically shown to achieve faster rate of convergence in (strongly) monotone games, e.g., \cite{Bo_LP_TechNote19}. One possibility of explaining this disparity is by considering \textit{relatively smooth} regularizers, which we leave for future work. Finally another open question is how these continuous-time results relate to their discrete-time and stochastic approximation counter-parts. 

\section{Appendix}
\label{sec:appendix}
\begin{lemma}(Proposition 2 of \cite{Bo_LP_TechNote19})
	\label{lem:mirror_map_properties}
	Let $\psi_\epsilon^p\! \!= \!\epsilon \vartheta^p\!$, $\epsilon \!>\!\! 0$, where  $\vartheta^p$ satisfies   \cref{assump:primal}, and let ${\psi_\epsilon^p}^\star$ be the convex conjugate of $\psi_\epsilon^p$. 
	Then, 
	\begin{enumerate}
		\item[(i)] ${\psi_\epsilon^p}^\star\!:\!\mathbb{R}^{n_p} \!  \to \!\mathbb{R}\cup\{\infty\}$ is closed, proper, convex and finite-valued over $\mathbb{R}^{n_p}$, i.e., $\dom({\psi_\epsilon^p}^\star) = \mathbb{R}^{n_p}$.
		\item[(ii)]  ${\psi_\epsilon^p}^\star$ is continuously differentiable and $\nabla {\psi_\epsilon^p}^\star \!= \!C^p_\epsilon$. 
		\item[(iii)] $C^p_\epsilon$ is  $(\epsilon\rho)^{-1}$-Lipschitz on $ \mathbb{R}^{n_p}$. \
		\item[(iv)]  $C^p_\epsilon\!$ is surjective from $\mathbb{R}^{n_p}\!$ onto $\rinterior(\Omega^p)$ whenever $\psi_\epsilon^p\!$ is steep, and  onto $\Omega^p$ whenever $\psi_\epsilon^p\!$ is non-steep.  
	\end{enumerate}
\end{lemma}

The following results will make heavy use of several well-known properties of the Bregman divergence (and their minor extensions), which can be found in a variety of references such as \cite{Amari16, Beck17}.

\begin{lemma}
	\label{lem:prop_bregman}
	Let $h: \mathbb{R}^n \to \mathbb{R} \cup \{\infty\}$ be a proper, closed, convex function, then, 
	\begin{itemize}
		\item[(i)] $D_h(x,y) \geq 0$ and equals $0$ if and only if $x = y$ whenever $h$ is strictly convex.
		\item[(ii)] Let $h^\star$ be the convex conjugate of $h$, then $D_{h^\star}(z^\prime, z) = D_h(x, x^\prime)$, where $z \in \partial h(x), z^\prime \in \partial h(x^\prime), x \in \partial h^\star(z), x^\prime \in \partial h^\star(z^\prime).$
		\item[(iii)] $D_h(x,y) + D_h(y,x) = (x-y)^\top(\nabla h(x) - \nabla h(y)), \forall x,y \in \dom(\partial h).$
	\end{itemize}
\end{lemma}

\begin{proof}(\textbf{Proof of \cref{prop:U_strongly_vs_relative}})
	Using \cref{lem:prop_bregman}(iii),  $D_h(x, x^\prime) + D_h(x^\prime, x) \!= (x - x^\prime)^\top(\nabla h(x) - \nabla h(x^\prime)), \forall x, x^\prime \in \dom(\partial h).$ Then \textit{(i)} follows from, $l\|x - x^\prime\|_2^2 \geq (x - x^\prime)^\top(\nabla h(x) - \nabla h(x^\prime))$, $\eta\|x \!-\! x^\prime\|^2_2   \geq (U(x) \!-\! U(x^\prime))\!^\top\!(x\!-\!x^\prime)$ ($\mathcal{G}$ is $\eta$-strongly monotone), and \textit{(ii)} follows from,  $(x - x^\prime)^\top(\nabla h(x) - \nabla h(x^\prime)) \geq \rho \|x-x^\prime\|_2^2$ and  $\mu\|x \!-\! x^\prime\|^2_2 \geq (U(x) \!-\! U(x^\prime))\!^\top\!(x\!-\!x^\prime)$ ($\mathcal{G}$ is $\mu$-hypo-monotone). 
\end{proof}

\begin{proof}(\textbf{Proof of \cref{thm:convergence_MD_relative_strongly_monotone}})\label{proof:theorem_relative_sc_monotone} By \cite[Theorem 1]{mertikopoulos2017convergence}, the strategies $x(t)$ generated by MD converges globally to the unique interior NE $x^\star \in \rinterior(\Omega)$.  Consider the Lyapunov function,
	\begin{equation}
		\label{eqn:Lyapunov_function_monotone_proof}
		\textstyle V_{z}(t) = \gamma^{-1} \sum_{p \in \mathcal{N}} D_{{\psi_\epsilon^p}^\star}(z^p,  {z^p}^\star),
	\end{equation}
	where  $D_{{\psi_\epsilon^p}^\star}$ is the Bregman divergence of ${\psi_\epsilon^p}^\star$. The rest point conditions \eqref{eqn:MD_rest_points} $N_\Omega(x^\star) \ni   U(x^\star)$ implies $U(x^\star) - n(x^\star) = \mathbf{0}$, for any normal vector $n(x^\star) \in N_\Omega(x^\star)$. Taking the time-derivative of $V_z$ along the solutions of MD and using \cref{lem:mirror_map_properties}, $C_\epsilon = (\nabla {\psi_\epsilon^p}^\star)_{p \in \mathcal{N}}$, $x = C_\epsilon(z)$,  $x^\star = C_\epsilon(z^\star)$,
	\begin{align*}
		\dot V_{z}(t) & =  (C_\epsilon(z) - C_\epsilon({z}^\star))^\top U(x) \\
		& = (x-x^\star)^\top(U(x) - U(x^\star) + n(x^\star)) \\
		& = (x-x^\star)^\top(U(x) - U(x^\star)) + (x-x^\star)^\top n(x^\star)\\
		& \leq (x-x^\star)^\top(U(x) - U(x^\star))\\
		& \leq -\eta (D_h(x, x^\star)\! +\! D_h(x^\star, x)),
	\end{align*}
	where we used the definition of a normal vector and $\eta$-relative strong monotonicity of $\mathcal{G}$.
	Using $D_h(x,x^\star) \geq 0$, $D_h(x^\star, x) = \sum_{p \in \mathcal{N}} D_{\vartheta^p }({x^p}^\star, x^p) = \epsilon^{-1} \sum_{p \in \mathcal{N}} D_{\psi_\epsilon^p}({x^p}^\star, x^p)\!=\! \epsilon^{-1}\!\sum_{p \in \mathcal{N}} D_{\psi_\epsilon^p}({x^p}^\star, x^p)$ and $D_{\psi_\epsilon^p}({x^p}^\star, x^p) = D_{{\psi_\epsilon^p}^\star}(z^p, {z^p}^\star)$ (\cref{lem:prop_bregman}(ii)), 
	\begin{align*}
		\dot V_{z}(t) & \textstyle  \leq -\eta D_h(x^\star, x) \leq -\eta\epsilon^{-1} \sum_{p \in \mathcal{N}} D_{{\psi_\epsilon^p}^\star}({z^p}, {z^p}^\star) = -\gamma \eta \epsilon^{-1}  V_z(t),
	\end{align*}  
	hence $V_{z}(t)\!\leq\! e^{-\gamma \eta \epsilon^{-1} t} V_{z}(0)$, which in turn implies  \eqref{eqn:MD_iterates_relative_strongly_monotone_Bregman}. \eqref{eqn:MD_iterates_relative_strongly_monotone} then follows from $D_{\vartheta^p}(x^p, {x^p}^\star) \geq \dfrac{\rho}{2}\|x^p- {x^p}^\star\|_2^2$.
\end{proof}

\begin{proof}(\textbf{Proof of \cref{prop:pseudo_gradient_rel_hypo_monotone}})
	\begingroup
	\allowdisplaybreaks
	Using the definition of $\widetilde U$, we have,
	\begin{align*} 
		(x-x^\prime)^\top(\widetilde U(x) - \widetilde U(x^\prime)) & = (x-x^\prime)^\top(U(x) - U(x^\prime)) - (x-x^\prime)^\top(\Psi_\epsilon(x) - \Psi_\epsilon(x^\prime))\\
		& \leq \mu (D_h(x, x^\prime)\! +\! D_h(x^\prime, x)) - (x-x^\prime)^\top(\Psi_\epsilon(x) - \Psi_\epsilon(x^\prime))\\
		& =  \mu (D_h(x, x^\prime)\! +\! D_h(x^\prime, x)) - \epsilon  (D_h(x, x^\prime)\! +\! D_h(x^\prime, x))\\
		& =  -(\epsilon - \mu) (D_h(x, x^\prime)\! +\! D_h(x^\prime, x)), \epsilon > \mu,
	\end{align*}
	where the first inequality follows from $\mu$-relative hypo-monotonicity and the equality immediate after uses $(x-x^\prime)^\top(\Psi_\epsilon(x) - \Psi_\epsilon(x^\prime)) = \epsilon  (D_h(x, x^\prime)\! +\! D_h(x^\prime, x))$, which  can be shown through a straightforward calculation (see \cref{lem:prop_bregman}(iii)).  By \cref{def:relative_monotone}, $\widetilde{\mathcal{G}}$ is $(\epsilon-\mu)$-relatively strongly monotone whenever $\epsilon > \mu$. 
	\endgroup
\end{proof}

\begin{proof}(\textbf{Proof of \cref{thm:convergence_DMD_relative_hypo_monotone}}) By \cite[Theorem 1]{Bo_LP_TechNote19}, the strategies $x(t)$ generated by DMD converges globally towards the unique interior perturbed NE $\overline x \in \rinterior(\Omega)$. 
	To derive the rate of convergence, we begin by showing that DMD can be transformed into an equivalent \textit{undiscounted} dynamics and then apply the same approach as in the MD case. 
	 Using $C_\epsilon = (C^p_\epsilon)_{p \in \mathcal{N}}$, it follows $x =  C_\epsilon(z) \Rightarrow z \in (C_\epsilon)^{-1}(x)$, where $(C_\epsilon)^{-1}(x) = (({C^p_\epsilon})^{-1}(x^p))_{p \in \mathcal{N}}=  (\partial \psi_\epsilon^p(x^p))_{p \in \mathcal{N}} =  (\nabla \psi_\epsilon^p(x^p) + N_{\Omega^p}(x^p))_{p \in \mathcal{N}} = \Psi_\epsilon(x) + N_{\Omega}(x)$, where $\Psi_\epsilon \coloneqq (\nabla \psi_\epsilon^p)_{p \in \mathcal{N}}$ and $N_{\Omega} \coloneqq (N_{\Omega^p})_{p \in \mathcal{N}}$, $N_{\Omega^p}$ is the normal cone of $\Omega^p$. This allows us to obtain an equivalent expression of DMD as an undiscounted dynamics whereby the pseudo-gradient is subjected to regularization,
	\begin{equation}
		\label{eqn:reformulated_DMD}
		\dot z  \in \gamma(U - \Psi_\epsilon - N)(x), \quad x =  C_\epsilon(z), 
	\end{equation} 
	or equivalently, let $n(x) \in N(x)$, then,
	\begin{equation}
		\label{eqn:reformulated_DMD_2}
		\dot z  =  \gamma(U - \Psi_\epsilon)(x) - \gamma n(x), \quad x =  C_\epsilon(z).
	\end{equation} 
	Letting  $\widetilde U \coloneqq  U - \Psi_\epsilon$, then by  \cref{lem:equilibria}, we see that $\widetilde U$ is the pseudo-gradient associated with the perturbed game with payoffs given by \eqref{eqn:perturbed_payoff}. 
	
	We now proceed to show the rate of convergence by employing the same Lyapunov function as \eqref{eqn:Lyapunov_function_monotone_proof}, except that we replace $z^\star$ by $\overline z$, where $\overline x = C_\epsilon(\overline z)$. Using \cref{lem:mirror_map_properties}(i), $C_\epsilon = (\nabla {\psi_\epsilon^p}^\star)_{p \in \mathcal{N}}$, $x = C_\epsilon(z)$,  taking the time-derivative of $V_z$ along the solutions of DMD,
	\begin{align*}
		 \dot V_{z}(t) &  = (C_\epsilon(z) - C_\epsilon(\overline z))^\top \dot z =   (C_\epsilon(z) - C_\epsilon(\overline z))^\top( \widetilde{U}(x) - n(x)).
	\end{align*}
	Subtracting rest point condition of \eqref{eqn:reformulated_DMD_2}, $(N_\Omega + (C_\epsilon)^{-1})(\overline x) \ni U(\overline x) \implies \mathbf{0} = \widetilde{U}(\overline x) - n(\overline x)$ on the right-hand side and using the monotonicity of the normal cone \cite{Beck17}, we obtain,
	\begin{align*}
		 \dot V_{z}(t)	& = (x-\overline x)^\top(\widetilde{U}(x) - \widetilde{U}(\overline x)) - (x - \overline x)^\top(n(x) - n(\overline x))\\
		& \leq (x-\overline x)^\top(\widetilde{U}(x) - \widetilde{U}(\overline x)) \leq -(\epsilon - \mu) D_h(\overline x, x),
	\end{align*} where the final inequality follows from \cref{prop:pseudo_gradient_rel_hypo_monotone} and $D_h(x, \overline x) \geq 0$. Then,
	\begin{align*}
		&  \dot V_{z}(t)  \textstyle \leq - (\epsilon - \mu)\epsilon^{-1} \sum_{p \in \mathcal{N}} D_{{\psi_\epsilon^p}}(\overline{x}^p, x^p) = -\gamma(\epsilon - \mu)\epsilon^{-1}  V_z(t), \epsilon > \mu,	
	\end{align*}
	which follows from $\textstyle D_h(\overline x, x) = \sum_{p \in \mathcal{N}} D_{{\vartheta^p}}( \overline x^p, x^p) =  \epsilon^{-1} \sum_{p \in \mathcal{N}} D_{{\psi_\epsilon^p}}( \overline x^p, x^p)$, and $D_{{\psi_\epsilon^p}^\star}(z^p, \overline z^p) \!=\! D_{{\psi_\epsilon^p}}(\overline x^p, x^p)$ (\cref{lem:prop_bregman}(ii)). Then \cref{eqn:DMD_iterates_rel_hypomonotone_Bregman} follows from  $V_z(t) \leq -\gamma(\epsilon - \mu)\epsilon^{-1} V_z(0)$. The rate for the null monotone case can be directly obtained from above by plugging in $\mu = 0$. 
\end{proof}

\begin{proof}(\textbf{Proof of \cref{thm:convergence_AC_relative_sc}}) Let $x^\star \in \interior(\Omega)$ the unique interior NE and consider, 
	\begin{align*}
		\textstyle V_{x,z}(t)  = \epsilon e^{ \gamma\eta \epsilon^{-1} t}  ( r^{-1}(P(x^\star) - P(x)) + \gamma^{-1}\sum_{p \in \mathcal{N}}  D_{{\psi_\epsilon^p}^\star}(z^p, {z^p}^\star)),
	\end{align*} 
	where $D_{{\psi_\epsilon^p}^\star}$ is the Bregman divergence of ${\psi_\epsilon^p}^\star$. Using \cref{lem:mirror_map_properties}, $C_\epsilon = (\nabla {\psi_\epsilon^p}^\star)_{p \in \mathcal{N}}$, $x^\star = C_\epsilon(z^\star)$, $\dot z = \gamma U(x) = \gamma \nabla P(x)$, taking the time-derivative of $V_{x,z}$ along the solutions of AC, we obtain,
	\begingroup
	\allowdisplaybreaks
	\begin{flalign*}
	 & \dot V_{x,z}(t)
	= \textstyle \eta \gamma e^{\gamma \eta \epsilon^{-1} t}  (r^{-1}(P(x^\star) - P(x))  + \gamma^{-1} \sum_{p \in \mathcal{N}} D_{{\psi_\epsilon^p}^\star}(z^p, {z^p}^\star)) &\\ 
	+ & \, \epsilon  e^{\gamma \eta \epsilon^{-1} t}(r^{-1}(-(\nabla P(x))^\top \dot x) + \gamma^{-1} (C_\epsilon(z) - C_\epsilon(z^\star))^\top \dot z) \\
	= & \, \textstyle   \eta \gamma e^{\eta \gamma\epsilon^{-1}  t}  (r^{-1}(P(x^\star) - P(x))  + \gamma^{-1} \sum_{p \in \mathcal{N}} D_{{\psi_\epsilon^p}^\star}(z^p, {z^p}^\star)) \\ 
	+ & \, \epsilon e^{\gamma \eta \epsilon^{-1} t}(-\nabla P(x)^\top (C_\epsilon(z) - x) + (C_\epsilon(z) - C_\epsilon(z^\star))^\top \nabla P(x))\\
	= & \, \textstyle  \eta e^{\gamma \eta \epsilon^{-1} t}  (\gamma r^{-1}(P(x^\star) - P(x))  + \sum_{p \in \mathcal{N}} D_{{\psi_\epsilon^p}^\star}(z^p, {z^p}^\star)) \\ 
	+ & \, \epsilon e^{\gamma \eta \epsilon^{-1} t} (-\nabla P(x)^\top(C_\epsilon(z) - C_\epsilon(z^\star) + x^\star - x)  +  (C_\epsilon(z) - C_\epsilon(z^\star))^\top \nabla P(x))\\
	= & \, \textstyle  \eta e^{\gamma \eta \epsilon^{-1} t}  (\gamma r^{-1}(P(x^\star)- P(x))  + \sum_{p \in \mathcal{N}} D_{{\psi_\epsilon^p}}({x^p}^\star, x^p)) \! -\!   \epsilon e^{\gamma \eta \epsilon^{-1} t} (\nabla P(x)^\top(x^\star - x))\\	
	\leq & \, \textstyle  \eta e^{\gamma \eta \epsilon^{-1} t} (\gamma r^{-1}(P(x^\star) - P(x))  + \sum_{p \in \mathcal{N}} D_{{\psi_\epsilon^p}}({x^p}^\star, x^p)) \\
	+ & \, \textstyle \epsilon e^{\gamma \eta \epsilon^{-1} t}(P(x) - P(x^\star) - \eta \sum_{p \in \mathcal{N}} D_{{\vartheta^p}}({x^p}^\star, {x^p}))\\	
	= & \, e^{\gamma \eta \epsilon^{-1} t}(\epsilon-\dfrac{\eta\gamma }{r})(P(x) - P(x^\star)) < 0, \forall x \neq x^\star,  \epsilon  > \frac{\eta\gamma}{r},
	\end{flalign*}
	\endgroup 
	where the inequality follows from from $\eta$-relative strong concavity of $P$  with respect to $h(x) = \sum_{p \in \mathcal{N}} \vartheta^p(x^p)$ (\cref{def:potential_relatively_concave}), $D_{{\vartheta^p}^\star}(z^p, {z^\star}^p) \!=\! D_{{\vartheta}^p}({x^p}^\star, x^p)$ (\cref{lem:prop_bregman}(ii)), and the last line follows from $\epsilon \sum_{p \in \mathcal{N}} D_{{\vartheta^p}}({x^p}^\star, x^p) =  \sum_{p \in \mathcal{N}} D_{{\psi_\epsilon^p}}({x^p}^\star, x^p) $. 
	
	Then \eqref{eqn:AC_potential_relative_sc} follows from $V_{x, z}(t)\leq V_{x, z}(0)$, $\sum_{p \in \mathcal{N}} D_{{\psi_\epsilon^p}^\star}(z^p, {z^p}^\star) \geq 0$, \eqref{eqn:thm_AC_iterates_relative_sc_bregman} follows from $P(x^\star) - P(x) \geq \eta \sum_{p \in \mathcal{N}} D_{{\vartheta^p}}(x^p, {x^p}^\star) = \eta D_{h}(x, {x}^\star) $ and  \eqref{eqn:thm_AC_iterates_relative_sc_euclidean} follows from $D_{\vartheta^p}(x^p, {x^p}^\star) \geq \dfrac{ \rho}{2}\|x^p- {x^p}^\star\|_2^2$ whenever  $\vartheta^p$ is $\rho$-strongly convex. 
\end{proof}

\newpage
\bigskip\bigskip\bigskip
\bibliographystyle{plainurl}
\bibliography{References}


\end{document}